\documentclass[journal,twoside,web]{ieeecolor}

\usepackage{etoolbox}
\makeatletter
\@ifundefined{color@begingroup}
{\let\color@begingroup\relax
\let\color@endgroup\relax}{}
\def\fix@ieeecolor@hbox#1{
\hbox{\color@begingroup#1\color@endgroup}}
\patchcmd\@makecaption{\hbox}{\fix@ieeecolor@hbox}{}{\FAILED}
\patchcmd\@makecaption{\hbox}{\fix@ieeecolor@hbox}{}{\FAILED}

\usepackage{generic}
\usepackage{cite}
\usepackage{amsmath,amssymb,amsfonts}
\usepackage{algorithmic}
\usepackage{graphicx}
\usepackage{textcomp}

\def\BibTeX{{\rm B\kern-.05em{\sc i\kern-.025em b}\kern-.08em
    T\kern-.1667em\lower.7ex\hbox{E}\kern-.125emX}}
\usepackage{graphicx}
\usepackage{graphics}
\usepackage{bm}
\usepackage{cite}
\usepackage{amsmath}
\usepackage{amssymb}
\usepackage{mathrsfs}
\usepackage{txfonts}
\usepackage{subfigure}
\usepackage{wrapfig}
\usepackage[citecolor=blue,linkcolor=blue,colorlinks=true]{hyperref}

\usepackage{tabls}
\newtheorem{theorem}{THEOREM}[section]
\newtheorem{definition}[theorem]{Definition}

\newtheorem{remark}[theorem]{Remark}

\newtheorem{lemma}[theorem]{Lemma}

\newtheorem{example}[theorem]{Example}

\begin{document}

\title{\huge{A geometric approach for stability analysis of delay systems——Applications to asymmetric network dynamics}}

\author{Shijie Zhou,~Luan Yang,~Xuzhe Qian, ~and~Wei Lin,~\IEEEmembership{Senior Member,~IEEE}
\thanks{Shijie Zhou and Luan Yang are both with the Research Institute of Intelligent Complex Systems, Fudan University, Shanghai 200433, China.}
\thanks{Xuzhe Qian is with the School of Mathematical Sciences, Fudan University, Shanghai 200433, China.}
\thanks{Shijie Zhou, to whom correspondence should be addressed. Tel. +86-21-55665141. Fax. +86-21-65646073.  E-mail: \url{sjzhou14@fudan.edu.cn}.}
\thanks{Wei Lin is with the Research Institute of Intelligent Complex Systems, the School of Mathematical Sciences, and the Shanghai Center for Mathematical Sciences, Fudan University, Shanghai 200433, China. E-mail: \url{wlin@fudan.edu.cn}.}}

\maketitle

\begin{abstract}
Investigating the network stability or synchronization dynamics of multi-agent systems with time delays is of significant importance in numerous real-world applications. Such investigations often rely on solving the transcendental characteristic equations (TCEs) obtained from  linearization of the considered systems around specific solutions. While stability results based on the TCEs with real-valued coefficients induced by symmetric networks in time-delayed models have been extensively explored in the literature, there remains a notable gap in stability analysis for the TCEs with complex-valued coefficients arising from asymmetric networked dynamics with time delays.  To address this challenge comprehensively, we propose a rigorously geometric approach.  By identifying and studying the stability crossing curves in the complex plane, we are able to determine the stability region of these systems.    This approach is not only suitable for analyzing the stability of models with discrete time delays but also for models with various types of delays, including distributed time delays. \textcolor{black}{Additionally, it can also handle random networks.} We demonstrate the efficacy of this approach in designing delayed control strategies for car-following systems, mechanical systems, and deep brain stimulation modeling, where involved are complex-valued TCEs or/and different types of delays. All these therefore highlight the broad applicability of our approach across diverse domains. \end{abstract}


\IEEEpeerreviewmaketitle

\section{Introduction and notations}

Long-time behaviors of complex dynamical systems has been a subject of intensive research. Various types of theories describing such behaviors were developed systematically, including the Lyapunov stability theory,  LaSalle's invariance principle and its variants 
\cite{b29,b30,b31}, the center manifold and bifurcation theories\cite{b32,b33,b34}, and the chaos theory.   Time delay, an inherent characteristic of real-world systems owing to the physical distance signals must travel, often leads to a more diversity of long-time behaviors\cite{b21,b42,b43}.  This naturally triggered an extensive amount of studies on the dynamics of time-delay systems. Among the studies, an elementary aspect is the examination of stability in  linear time-delay systems, which was typically analyzed using frequency-domain approaches\cite{b2, b14}. These approaches involve the analysis of the transcendental characteristic equations (TCEs),  bringing challenges in both analytical and numerical manners. The seminal work, as outlined in Ref.~\cite{b14}, relied on the $\tau$-decomposition concept \cite{b79,b80,b81,b12,b13}, which can be seen as a particular instance of the $D$-partition approach \cite{b25, b26, b27}.  Since then, numerous researchers have made significant contributions.  For instance, the stability criteria for systems with large delays were established in Refs.~\cite{b22,b23}, while the absolute stability for systems with discrete-time delays were investigated in Refs.~\cite{b16,b17,b24}. Additionally, the geometric stability switch criteria for systems with delay-dependent parameters were developed in Refs.~\cite{b15,b18,b19}.  More recently, the classical frequency domain analysis also was extended for systems with time-varying delays in Refs.~\cite{b57,b58}.

The $\tau$-decomposition approach encompasses two crucial aspects:  The exhaustive identification of critical imaginary roots and the examination of asymptotic behavior of these critical imaginary roots. This approach was widely employed in the stability analysis of various discrete time-delay systems \cite{b2,b7,b8,b9,b14,b36,b37,b28,b47,b48,b38,b50,b51}. Furthermore, it was extended to study systems with uniformly distributed delay as well \cite{b4,b6}. Notably, a novel frequency-sweeping framework was put forth in the work of Ref.~\cite{b2}, leading to three noteworthy advancements: A more comprehensive classification scheme for time-delay systems that accounts for regular singularities, an introduction of a general invariance property, and a successful resolution of the entire stability problem. Moreover, new algebraic and geometric analyses emerged, specifically focusing on examining the dynamics of local stability crossing curves~\cite{b5}.

In spite of the above-mentioned advances and extensions, previous contributions primarily centered on examining the stability of quasi-polynomial TCEs in the form of $F(\lambda,\tau)=\sum_{n=0}^q{b_n(\lambda)}{\rm e}^{-n\tau\lambda}=0$, where $b_i(\lambda)$ with $0\leq i\leq q$ are polynomials of real-valued coefficients and with respect to $\lambda$.  Actually, these TCEs originate from discrete time-delay systems $\dot{\bm{z}}=\sum_{l=0}^m\bm{B}_l\bm{z}(t-l\tau)$, where $\bm{B}_l$ are real matrices\cite{b2,b8,b28}. 
However, in real applications, it is inadequate to only consider real matrices $\bm{B}_l$ and real-valued coefficients in $b_i(\lambda)$.  For example, in a multi-agent system (MAS) described by $\dot{\bm{x}_i}=\bm{Q}\bm{x}_i+\sum_{j=1}^Na_{ij}\bm{x}_j(t-\tau)$ for $i=1,\cdots,N$ and with the Laplacian network matrix $\bm{J}\triangleq\{a_{ij}\}_{N\times N}$, we often transform its consensus problem into a stability problem by decoupling the dynamics into the corresponding orthogonal manifold of lower dimensions using the master stability function as: $\dot{\bm{z}}_k=\bm{Q}\bm{z}_k+\mu_k \bm{z}_k(t-\tau)$ where $k=2,\cdots,N$ and $\mu_k$ led by $\mu_1=0$ are the eigenvalues of the matrix $\bm{J}$ (see details in Section \ref{MSF}).  Guaranteeing the stability as well as the consensus thus requires all the eigenvalues $\mu_k$ of $\bm{J}$ (except for $0$) to be located within the stability region 
\begin{equation}\label{1056}
\Omega\triangleq \left\{L\in\mathbb{C}~\bigg|~\dot{\bm{z}}=\bm{Q}\bm{z}+L\bm{z}(t-\tau)~\mbox{is stable} \right\}.
\end{equation}
Most of the previous studies have made significant contributions by assuming that the network matrix $\bm{J}$ is symmetric \textcolor{black}{and deterministic}\cite{b47,b50,b78}, so that all the eigenvalues $\mu_l$ are real \textcolor{black}{and obtainable}, simplifying the investigation on the set $\Omega\cap\mathbb{R}$.  However, it is practically necessary to consider the asymmetric and random network matrix $\bm{J}$, since most real networks, including the social networks\cite{b60}, the webpage links\cite{b61}, and the gene regulatory networks\cite{b62}, are not only asymmetric \textcolor{black}{but also exhibit uncertainty}. Therefore,  $\mu_k$, the nonzero eigenvalues, are not often simply real \textcolor{black}{and may not be computable explicitly}. This urges us to comprehensively investigate the stability region $\Omega$ across the entire complex plane.

Moreover, modeling only using the discrete time-delay systems easily neglects the potential memory effects in the system dynamics over a specific time interval\cite{b35}. In the context of the car-following system\cite{b45,b46} which investigated the movement patterns of individual cars within a system, the behaviors of these cars may rely on historical information distributed over a time interval (see Example \ref{carfollowing3}). Another representative example pertains to the logistic equation with distributed delays stemming from cell biology\cite{b55},  where the proliferative cells at a given time instant are precisely those cells that entered the proliferative subpopulation within a time interval. Thus, it is of practical significance to introduce distributed time delays into the modeling of the engineering, physical or/and biological systems; however, a comprehensive approach for analyzing this kind of time-delay systems, even with complex-valued coefficients, remains largely unaddressed.





This article, therefore, aims to develop a geometric approach for analyzing the stability of a general group of linear time-invariant systems where complex-valued parameters and various types of time delays are simultaneously are taken into account.  The approach to be developed will entail the identification and analysis of the stability crossing curves for assessing system stability.  To demonstrate this approach in control problems, we will use three representative examples and design the corresponding delayed control strategies to achieve consensus/stability or eliminate synchronization in time-delay systems with general structures of complex networks.

We  highlight two key contributions of this article as follows:
(1)  rigorous establishment of the geometric approach for the stability analysis of the linear time-invariant systems with both complex-valued coefficients and various types of delays, including \textit{discrete} delays and \textit{distributed} ones, and
(2) offer of valuable insights into the design of the appropriate strategies for controlling the time-delay dynamical complex systems with \textit{asymmetric} networks. 

As for the second above-summarized contribution, we provide an additional illustration.   In study of the consensus or stability problem of a specific MAS, the traditional $\tau$-decomposition approach is computationally demanding for large size $N$, where the stability analysis is conducted based on $N-1$ (or $N$) quasi-polynomials with different coefficients \cite{b47,b51,b45}.  In contrast, our approach only requires to compute the stability crossing curves, making it computationally efficient (see Example \ref{carfollowing3}, Fig. \ref{fig15}).   Moreover, our approach is well-suited for coping with \textit{random} networks, where the exact eigenvalues of the network matrix are difficult to obtain, but an approximation of the eigenvalue distribution is feasible (see Example \ref{MAS0}, Fig. \ref{figm}).




We proceed as follows. In Section \ref{sec2}, we lay out the fundamental concepts and preliminaries for this study. In Section \ref{sec2.5}, we establish the geometric approach, entailing the identification and analysis of stability crossing curves, to analyze the TCEs.   In Section \ref{sec5.3}, we employ the established approach to stability analysis for scalar delay differential equations (DDEs).  In Section \ref{sec5.5}, we provide several illustrative examples, including the consensus/stability of the MASs and the elimination of synchronization in coupled oscillators, to demonstrate the efficacy of the proposed geometric approach in applications. In Section \ref{sec6}, we present some concluding remarks and suggest possible future directions for further research.

\textbf{Notations}. In this article, we use standard notations and terminologies.  Specifically,  $\mathbb{R}$ (resp., $\mathbb{R}_{+}$, $\mathbb{R}_{-}$) denotes the set of all real (resp.,  positive, negative) numbers, and $\mathbb{C}_+$ (resp., $\mathbb{C}_-$, $\mathbb{C}_0$) represents the set of complex numbers with positive (resp., negative, zero) real parts.  Denote by $\overline{\mathbb{R}}_+\triangleq\mathbb{R}_+\cup\{0\}$ and by $\overline{\mathbb{C}}_+\triangleq\mathbb{C}_+\cup\mathbb{C}_0.$   Also denote by ${\rm i}=\sqrt{-1}$ the imaginary unit,  by ${\rm det}(\cdot)$ the determinant of a matrix, and by $\bm{I}_q$ the identity matrix of dimension $q$.  For $\lambda\in\mathbb{C}$, denote, respectively, by ${\rm Re}\lambda$, ${\rm Im} \lambda$, $|\lambda|$, and $\overline{\lambda}$ the real part, the imaginary part, the norm, and the conjugate number of $\lambda$. For $x\in\mathbb{R}$,  denote by ${\rm Sgn}(x)$ the sign of $x$, where ${\rm Sgn}(x)=1$, $0$, and $-1$, respectively, for $x>0$, $x=0$, and $x<0$.   As usual, $\mathbb{N}$ (resp., $\mathbb{N}^*$) is the set of non-negative (resp., positive) integers. For the function $F(\lambda,L)$ defined in \eqref{character}, denote by $\partial_{\lambda}F(\lambda,L)$ and $\partial_{L}F(\lambda,L)$ the partial derivative with respect to $\lambda$ and $L$, respectively.

\section{Preliminaries}\label{sec2}

Enlightened by \eqref{1056}, we consider a more general class of $q$-dimensional and time-delay systems, described by
\begin{equation}\label{2}
\dot{\bm{z}}=\bm{Q}(L)\bm{z}+\bm{B}(L)\int_0^{+\infty}\bm{z}(t-\tau)h(\tau){\rm d}\tau,
\end{equation}
where $\bm{z}(t)\in\mathbb{C}^q$ is the state variable, $\bm{Q}(L)$ and $\bm{B}(L)$ are both \textit{analytic} and $q\times q$ matrix-valued function with respect to the complex-valued parameter $L$,
and $h(\tau)$, the distribution of the time delay, belongs to the function family $\mathscr{F}([0,+\infty),\overline{\mathbb{R}}_+)\triangleq\mathscr{D}([0,+\infty),\overline{\mathbb{R}}_+)\cup\mathscr{M}([0,+\infty),\overline{\mathbb{R}}_+)$. Here, $\mathscr{D}([0,+\infty),\overline{\mathbb{R}}_+)$ comprises of all the Dirac delta functions on $[0,+\infty)$, and $\mathscr{M}([0,+\infty),\overline{\mathbb{R}}_+)$ comprises of all the Borel measurable nonnegative functions on $[0,+\infty)$ normalized by $\int_0^{+\infty}h(\tau){\rm d}\tau=1$.

\begin{remark}
To enhance readability, we provide a few examples of $h(\tau)$ here to illustrate system \eqref{2}. For the sake of simplicity, we assume as $\bm{Q}(L)\equiv \bm{Q}$ and $\bm{B}(L)=L\bm{I}_q$.  For example, when $h(\tau)\in\mathscr{D}([0,+\infty),\overline{\mathbb{R}}_+)$ is specified as $\delta(\tau-\tau')$, system \eqref{2} becomes $\dot{\bm{z}}=\bm{Q}\bm{z}+L\bm{z}(t-\tau')$, which actually is the discrete time-delay system considered in \eqref{1056} and
the simplest case that has been extensively investigated in the past literature.  On the other hand, $h(\tau)\in\mathscr{M}([0,+\infty),\overline{\mathbb{R}}_+)$ corresponds to the system with distributed delays. One typical example is 
$$
h(\tau)=\begin{cases}{1}/{A}, & \tau \in[a, a+A], \\ 0, & \tau\notin[a, a+A],\end{cases}
$$ 
which corresponds to uniformly distributed delays. Another common example is $h(\tau)=\frac{n^n}{(n-1)! T^n}\tau^{n-1}{\rm e}^{-\frac{\tau n}{T}}$, which corresponds to the Gamma distributed delays.
\end{remark}

Our objective is to describe explicitly the stability region $\Omega$ for system \eqref{2} which is defined as follows.

\begin{definition}
System \eqref{2} is said to be \textbf{stable}, if $\lim_{t\to+\infty}\bm{z}(t)=\bm{0}$ for all solutions of system \eqref{2}. 
\end{definition}
\begin{definition}
The \textbf{stability region} for system \eqref{2} is denoted by
$\Omega\triangleq \left\{L\in\mathbb{C}~\big| ~\mbox{System \eqref{2} is stable} \right\}$.
\end{definition}

We give the definition of characteristic function and characteristic equation as follows (refer to \cite[Chapter 7, Lemma 2.1]{b21}).
\begin{definition}
The \textbf{characteristic function} of system \eqref{2} is given by 
\begin{equation}\label{character}
F(\lambda,L)\triangleq {\rm det}\left[\lambda \bm{I}_q-\bm{Q}(L)-\bm{B}(L)\int_0^{+\infty}{\rm e}^{-\lambda \tau}h(\tau){\rm d}\tau\right].
\end{equation}
The equation $F(\lambda,L)=0$ with respect to complex variable $\lambda$ is said to be the \textbf{characteristic equation} of system \eqref{2}.
 \end{definition}
 
Note that the characteristic equation of system \eqref{2} 
 does contain the exponential function with respect to $\lambda\in\mathbb{C}$.  Thus, we also call this equation as a TCE as mentioned above.
 
 \begin{definition}
Denote by ${\rm NU}(L)\in\mathbb{N}\cup\{+\infty\}$ the number of the roots of the TCE $F(\lambda,L)=0$ in $\overline{\mathbb{C}}_{+}$. We refer to these roots as \textbf{unstable roots}.
\end{definition} 

According to \cite[Chapter 7, Corollary 6.1]{b21}, we know that the system \eqref{2} is stable if and only if there is no root of the TCE $F(\lambda,L)=0$ in $\overline{\mathbb{C}}_+$.  Therefore, the stability region for system \eqref{2} must satisfy $\Omega=\left\{L\in\mathbb{C}~\big|~{\rm NU}(L)=0\right\}$.

\begin{definition}
For a region $\Theta\subset \mathbb{C}$, if ${\rm NU}(L)=a$ for all $L\in\Theta$, then we say ${\rm NU}(\Theta)=a$.
\end{definition}

\section{Geometric approach to analyze the characteristic equation}\label{sec2.5}

It is challengeable to solve the TCE $F(\lambda,L)=0$ either numerically or theoretically. To address this challenge, we develop a geometric approach to analyze it. We study how ${\rm NU}(L)$ changes with $L$ in $\mathbb{C}$. Intuitively speaking, as $L$ varies continuously, the value of ${\rm NU}(L)$ changes only when a root crosses over $\mathbb{C}_0$. Inspired by this speculation, we first give the definition of stability crossing curves and the definition of critical imaginary root, respectively.

\begin{definition}\label{SCCs}
Referring to system \eqref{2} or the characteristic function \eqref{character}, we denote by 
$$
\mathscr{A}\triangleq \left\{L\in\mathbb{C}~\Big|~\mbox{There exists}~\lambda\in\mathbb{C}_0~\mbox{such that}~F(\lambda,L)=0\right\}.
$$ 
This set can be locally parameterized as a curve $L=L(\beta)$, where $F({\rm i}\beta,L(\beta))=0$. Consequently, we refer to these parameterized curves as the \textbf{stability crossing curves (SCCs)}. Furthermore, for each $L\in\mathscr{A}$, we refer to the root $\lambda\in\mathbb{C}_0$ of the TCE $F(\lambda,L)=0$ as  the \textbf{critical imaginary root}.
\end{definition}

\subsection{Root continuity argument }\label{sec3}
In this subsection, we prove the root continuity argument, which states that the values of ${\rm NU}(L)$ only change at the SCCs. This theorem plays a fundamental role in developing the geometric approach.

\begin{figure}
\begin{center}
\centering\includegraphics[width=0.18\textwidth]{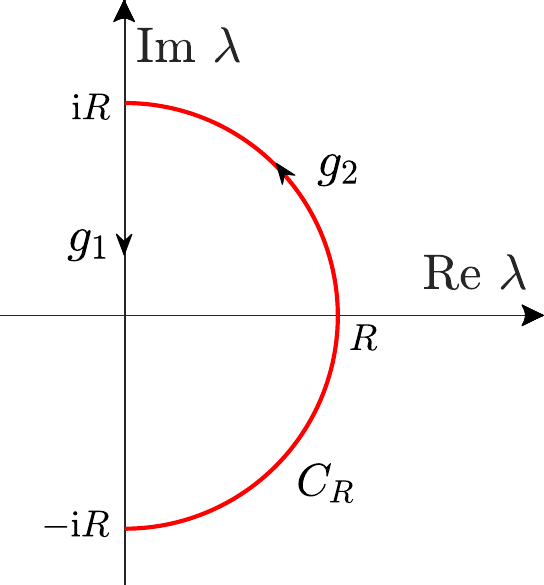}
\caption{According to Theorem \ref{theorem1}, all unstable roots of the TCE $F(\lambda,L)=0$ with $L\in\Theta$ are located within a semicircular bounded by the contour $C_R=g_1\cup g_2$ and vary continuously with respect to $L$. From this, it follows that the value of ${\rm NU}(L)$ changes only if a root crosses the imaginary axis $\mathbb{C}_0$, indicating that these changes only occur at the SCCs.} \label{fig11}
\end{center}
\end{figure}

\begin{theorem}\label{theorem1}
For a bounded connected region $\Theta\subset \mathbb{C}$, if $\mathscr{A}\cap\Theta=\emptyset$, then ${\rm NU}(L)$ keeps constant in $\Theta$.
\end{theorem}

\begin{proof} The characteristic function \eqref{character} can be expressed explicitly as
$$
\begin{aligned}
F(\lambda,L)&=\lambda^q-\sum_{0\leq k\leq q-1}^{k+j\leq q}P_{k,j}(L)\lambda^k\hat{h}(\lambda)^j \\
&=\lambda^q\left[1-\sum_{0\leq k\leq q-1}^{k+j\leq q}\lambda^{k-q}P_{k,j}(L)\hat{h}(\lambda)^j\right],
\end{aligned}
$$
where $\hat{h}(\lambda) \triangleq\int_0^{+\infty}{\rm e}^{-\lambda \tau}h(\tau){\rm d}\tau$ and $P_{k,j}(L)$ are all continuous functions with respect to $L$. Notice that, for $\lambda\in\overline{\mathbb{C}}_+$, $|\hat{h}(\lambda)|\leq \int_0^{+\infty}|h(\tau)|{\rm d}\tau=1$.   We thus obtain that
$$
\begin{aligned}
&\lim_{|\lambda|\to+\infty,~\lambda\in\overline{\mathbb{C}}_+}\left|\sum_{0\leq k\leq q-1}^{k+j\leq q}\lambda^{k-q}P_{k,j}(L)\hat{h}(\lambda)^j\right| \\
&\leq\lim_{|\lambda|\to+\infty,~\lambda\in\overline{\mathbb{C}}_+}\sum_{0\leq k\leq q-1}^{k+j\leq q}\left|\lambda\right|^{k-q}\left|P_{k,j}(L)\right|\left|\hat{h}(\lambda)\right|^j=0,
\end{aligned}
$$
where the limit is taken uniformly with respect to $L\in\Theta$. Thus, by choosing sufficiently large $R>0$, it follows that $|F(\lambda,L)|\geq |\lambda|^q/2>0$ for $\lambda\in\left\{\lambda\in\overline{\mathbb{C}}_+~\big|~|\lambda|>R\right\}$. This implies that all unstable roots of the TCE $F(\lambda,L)=0$ are located inside a semicircle $\left\{\lambda\in\overline{\mathbb{C}}_+~\big|~|\lambda|\leq R\right\}$.

According to the argument principle\cite[Chapter 3, Theorem 4.1]{b71}, we obtain that
\begin{equation}\label{argument}
{\rm NU}(L)=\dfrac{1}{2\pi{\rm i}}\int_{C_R}\dfrac{\partial_{\lambda}F(\lambda,L)}{F(\lambda,L)}{\rm d}\lambda.
\end{equation}
Here, $C_R$  represents the contour $g_1\cup g_2$ (see Fig. \ref{fig11}), and
$$
g_1\triangleq\Big\{\lambda={\rm i}\beta ~\Big|~ \beta:R\to -R \Big\}, ~
g_2\triangleq\left\{\lambda=R{\rm e}^{{\rm i}\theta}~\Big|~ \theta:-\dfrac{\pi}{2}\to \dfrac{\pi}{2} \right\}.
$$
Using the assumption $\mathscr{A}\cap\Theta=\emptyset$ yields $F(\lambda,L)\neq 0$ on $g_1$. This indicates that the term on the  right side of \eqref{argument} is continuous with respect to the variable $L$. This further implies that ${\rm NU}(L)$ is continuous with respect to $L$. Since ${\rm NU}(L)$ is an integer and $\Theta$ is connected, we conclude that ${\rm NU}(L)$ is a constant in the whole region $\Theta$. This therefore completes the proof.
\end{proof}

\begin{remark}\label{robust}
The {\it key} point in proving Theorem \ref{theorem1} is to establish an absence of the roots at infinity. This theorem is referred to as the root continuity argument because it demonstrates that all the roots in $\mathbb{C}_+$ vary continuously with respect to $L$ (see \cite[Lemma 2.1]{b15}, \cite[Proposition 3.1]{b18}). From this, it follows that \textit{the value of ${\rm NU}(L)$ only changes if a root appears on or cross $\mathbb{C}_0$ for some $L$}. This theorem holds significant importance for our geometric approach, as it suggests that 
\textit{the stability property of system \eqref{2} only changes at the SCCs}. Additionally, it also implies that the stability property is robust against small variations of the parameters, so that slight perturbations to parameters do not result in instability.
\end{remark}

\begin{figure}
\begin{center}
\centering\includegraphics[width=0.45\textwidth]{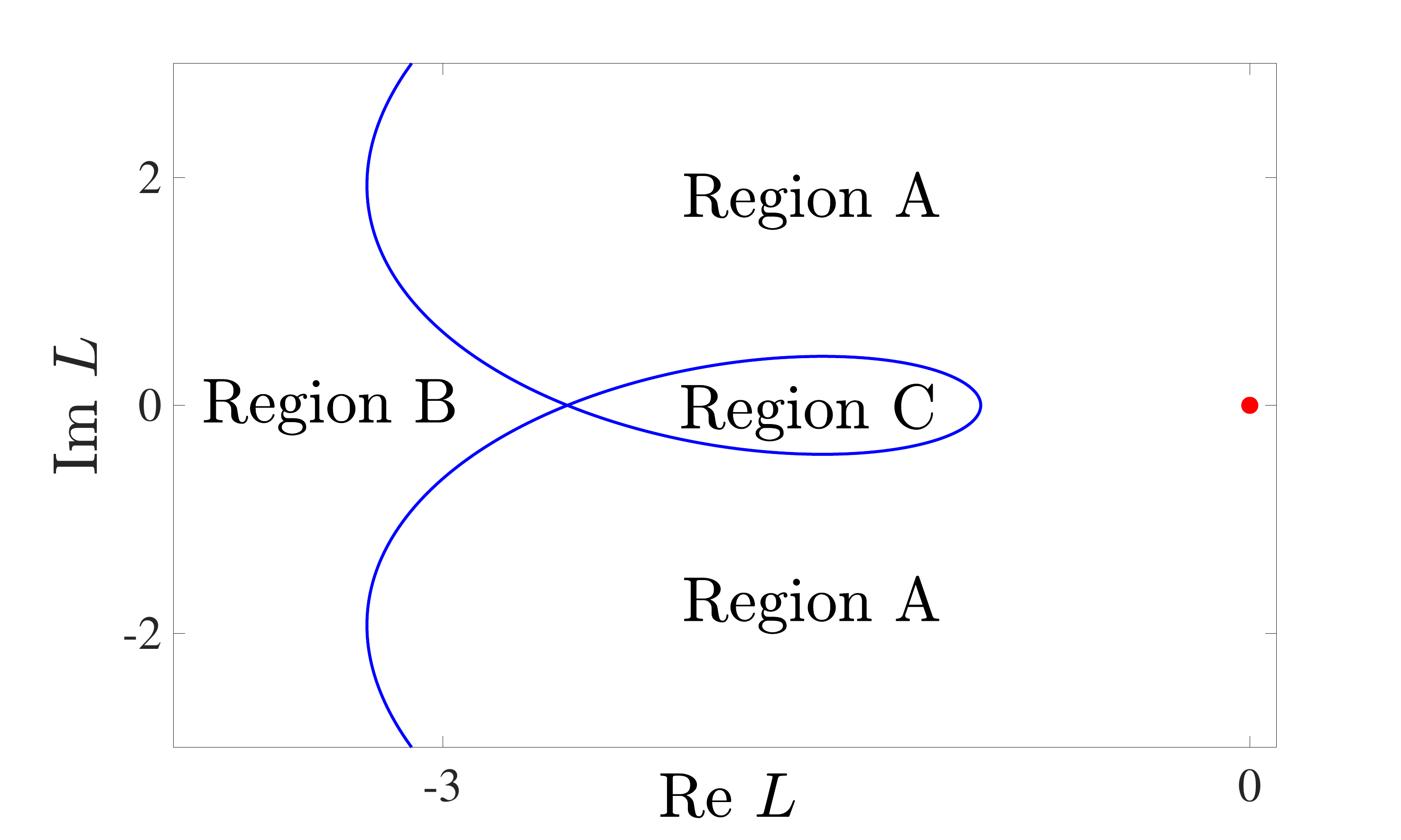}
\caption{The SCCs for system \eqref{1105} (the blue curves) separate the complex plane into several regions. The origin $L=0$, which is highlighted by a red dot, lies within Region $A$.  According to Theorem \ref{theorem1},  ${\rm NU}(0)=1$ implies ${\rm NU}({\rm Region}~A)=1$.} \label{fig1}
\end{center}
\end{figure}

\begin{example}
Consider a linear time-invariant system with a discrete time-delay as
\begin{equation}\label{1105}
\dot{z}=z+Lz(t-1/2),
\end{equation}
where $L$ is a complex-valued parameter.  The TCE for this system is $F(\lambda,L)\triangleq \lambda-1-L{\rm e}^{-\frac{\lambda}{2}}=0$.  As usual, we substitute $\lambda={\rm i}\beta$ into the equation.  Different from the traditional method, we focus the SCCs, parameterizing it as  $L=L(\beta)={\rm e}^{{\rm i}\beta/2}({\rm i}\beta-1)$.  As seen in Fig. \ref{fig1}, the SCCs separate the complex plane of $L$ into several regions.  According to Theorem \ref{theorem1}, NU sustains its value as constant within each of these regions. When $L=0$ (whose loci is highlighted by the red dot in Fig.~\ref{fig1}), the TCE becomes $\lambda-1=0$, from which we obtain ${\rm NU}(0)=1$.  Consequently, we have ${\rm NU}({\rm Region}~A)=1$, in which the connected Region $A$ indicated in Fig.~~\ref{fig1}.
\end{example}

\subsection{Geometric approach for establishing stability regions}\label{sec3.5}

In this subsection, we investigate the geometric property of the SCCs (see Theorem \ref{geometric} and Remark \ref{2106}), which comes from the asymptotic behavior of the critical imaginary  roots (see Lemma \ref{aba}). Upon this, we propose a geometric approach for establishing stability regions $\Omega$ for linear time-delay systems (see Example \ref{2107}).

First, we investigate the asymptotic behavior of the critical imaginary roots. Suppose a point $L^*$ to belong to the SCCs, given that $F(\lambda^*,L^*)=0$ and $\lambda^*\in\mathbb{C}_0$. A question arises ``What happens to the root of equation $F(\lambda,L)=0$, if we move $L$ from $L^*$ to its neighborhood?''
\begin{lemma}\label{aba}
Suppose that $F(\lambda^*,L^*)=0$ where $\lambda^*\in\mathbb{C}_0$. If $\partial_{\lambda} F(\lambda^*,L^*)\neq 0$,
then there exists an implicit function $\lambda(L)$ for $L$ in some neighborhood of $L^*$ such that 
$$
F(\lambda(L),L)=0, ~
\lambda(L^*)=\lambda^*, ~
\lambda'(L^*)=-\dfrac{\partial_{L}F(\lambda^*,L^*)}{\partial_{\lambda}F(\lambda^*,L^*)}.
$$
\end{lemma}

This Lemma can be validated directly using the well-known Implicit Function Theorem. Next, we investigate the geometric property of the SCCs. 

\begin{theorem}\label{geometric}
Suppose that $L=L(\beta^*)$ belongs to a local SCC which be locally parameterized as $L=L(\beta)$. Denote by $\vec{n}\triangleq{\rm i}\cdot {L'(\beta^*)}$ the normal vector of the local SCC at $L^*$. Therefore, ${\rm NU}(L^*+\epsilon \vec{n})-{\rm NU}(L^*-\epsilon \vec{n})=-1$ for sufficiently small $\epsilon>0$.
\end{theorem}
\begin{figure}
\begin{center}
\centering\includegraphics[width=0.45\textwidth]{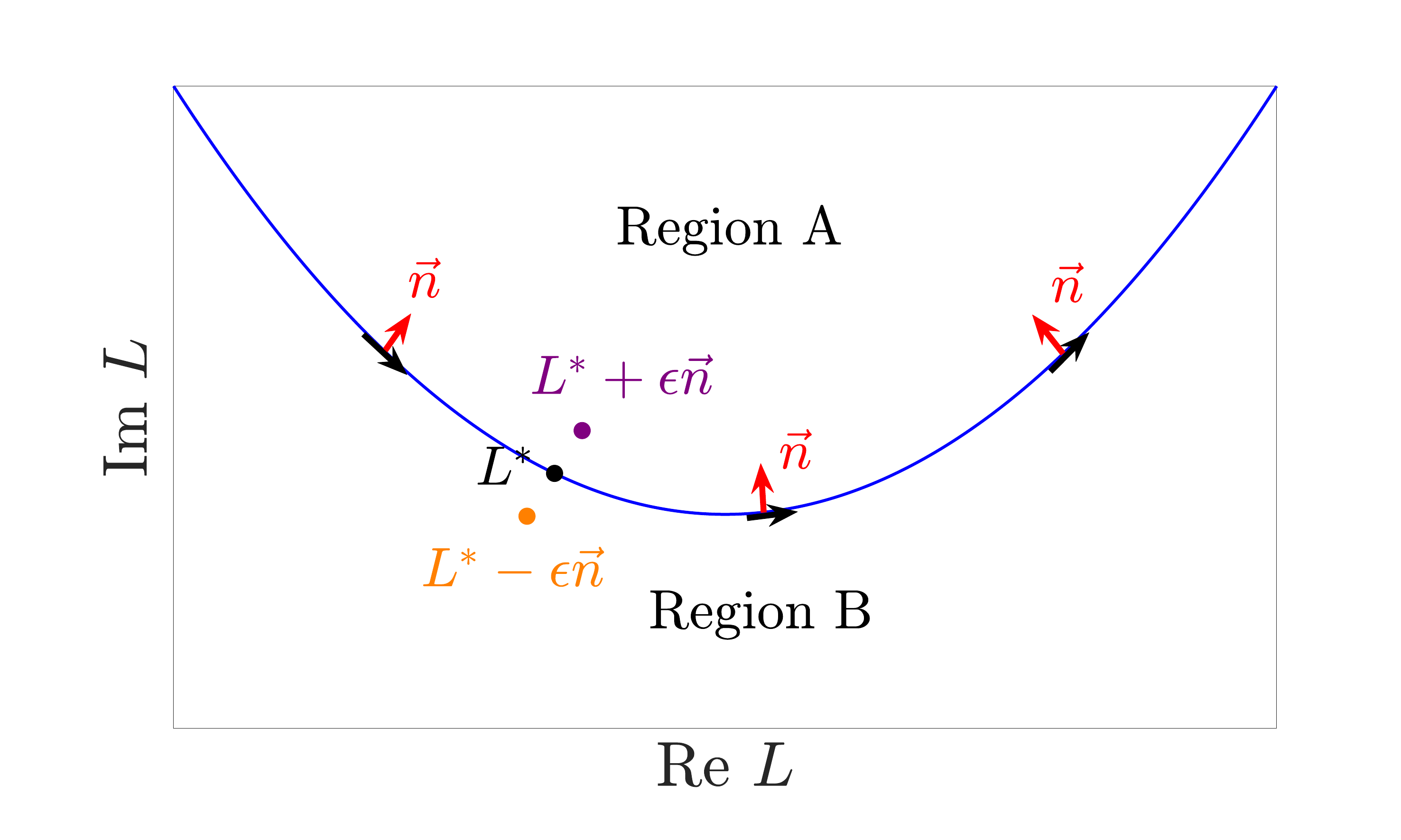}
\caption{The blue curve represents the local SCC $L(\beta)$, the black arrow represents the increasing direction of $\beta$, and $\vec{n}$ represents the normal vector defined in Theorem \ref{geometric}. The direction of $\vec{n}$ is obtained by a 90-degree-counterclockwise rotation of the tangent vector $L'(\beta)$. As proved in Theorem \ref{geometric}, when $L$ moves from $L^*-\epsilon\vec{n}$ (the brown dot) to $L^*+\epsilon\vec{n}$ (the purple dot), one root of equation $F(\lambda,L)=0$ undergoes a continuous transition from $\mathbb{C}_+$ to $\mathbb{C}_-$, crossing through $\mathbb{C}_0$ at $\lambda^*$ as $L=L^*$ (the black dot). Therefore, ${\rm NU}({\rm Region}~A)-{\rm NU}({\rm Region}~B)=-1$.} \label{fig3}
\end{center}
\end{figure}
\begin{proof}
For simplicity of denotations, we use $\partial_{\lambda}F$ and $\partial_L F$ instead of $\partial_{\lambda}F(\lambda^*,L^*)$ and $\partial_{L}F(\lambda^*,L^*)$, respectively.  According to Lemma \ref{aba}, we have
\begin{equation}
\label{yinhanshu0}
\lambda(L)-\lambda(L^*)=\lambda'(L^*)(L-L^*)+o(L-L^*),
\end{equation}
where
$\lambda'(L^*)=-{\partial_{L}F}\big/{\partial_{\lambda}F}$. 
By  substituting $L=L^*\pm\epsilon \vec{n}$ into Eq.~\eqref{yinhanshu0}, we deduce that
\begin{equation}\label{yinhanshu}
\lambda(L^*\pm\epsilon\vec{n})-{\rm i}\beta^*=\mp\epsilon\dfrac{\partial_{L}F}{\partial_{\lambda}F}\vec{n}+o(\epsilon).
\end{equation}
It follows from $F({\rm i}\beta, L(\beta))=0$ that $L'(\beta^*)=-{{\rm i}\cdot\partial_{\lambda}F}\big/{\partial_{L} F}$. Then, we derive  $\vec{n}={\rm i}{L'(\beta^*)}={\partial_{\lambda}F}\big/{\partial_L F}$. Substituting this result into Eq.~\eqref{yinhanshu} gives:
$$
\lambda(L^*\pm\epsilon\vec{n})={\rm i}\beta^*\mp{\epsilon}+o(\epsilon),
$$
which indicates that ${\rm Sgn}~[{\rm Re}\lambda(L^*\pm\epsilon\vec{n})]=\mp 1.$
This further implies that, \textit{as $L$ moves from $L^*-\epsilon \vec{n}$ to $L^*+\epsilon \vec{n}$, one of the roots of equation $F(\lambda,L)=0$, denoted by $\lambda(L)$, undergoes a continuous transition from $\mathbb{C}_+$ to $\mathbb{C}_-$, crossing through $\mathbb{C}_0$ at $\lambda^*$ when $L=L^*$} (see Fig. \ref{fig3}). Consequently, this implies that ${\rm NU}(L^*+\epsilon \vec{n})-{\rm NU}(L^*-\epsilon \vec{n})=-1.$
\end{proof}
\begin{remark}\label{2106}
The normal vector defined in Theorem \ref{geometric} can be expressed as $\vec{n}={\rm e}^{{\rm i}\frac{\pi}{2}}L'(\beta^*)$, which indicates that it is obtained by a 90-degree-counterclockwise rotation of the tangent direction of the SCCs with respect to the increase of $\beta$.  Thus, the geometric interpretation of Theorem \ref{geometric} is demonstrated as follows: \textit{Given the local representation of the SCC as shown in Fig. \ref{fig3}, where the  black arrow indicates the increasing direction of $\beta$, we  express that ${\rm NU}({\rm Region}~A)-{\rm NU}({\rm Region}~B)=-1$.} This characteristic offers a geometric approach for defining the stability region for a specific group of  complex-valued and linear time-delay systems. In the following, we provide an illustrative example.
\end{remark}

\begin{figure}
\begin{center}
\centering\includegraphics[width=0.45\textwidth]{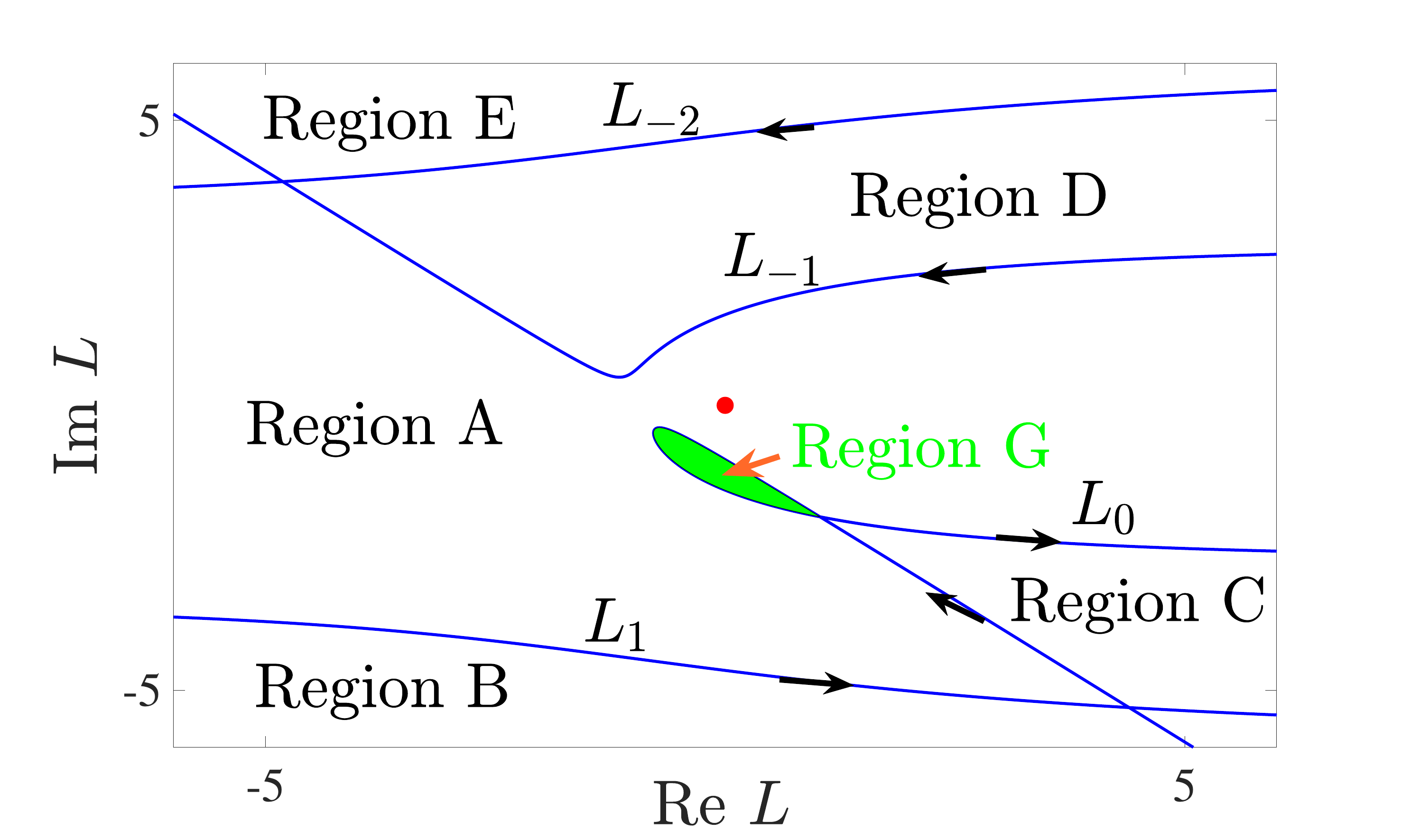}
\caption{The SCCs $L=L_{k}(\beta)$  for system \eqref{example2} separate the complex plane into several regions, while the black arrows represent the increasing directions of $\beta$. Here, the red dot represents the origin and from this, we deduce that  ${\rm NU}({\rm Region}~A)=1$. The value of ${\rm NU}$ for each region is easily obtained using Theorem \ref{geometric}, so that the stability region $\Omega$ for the considered system is ${\rm Region}~G$ (see the green shaded area).} \label{fig5}
\end{center}
\end{figure}

\begin{example}\label{2107}
Consider a complex-valued and linear time-delay system as
\begin{equation}\label{example2}
    \dot{z}=0.1(1+{\rm i})z+L(z(t-1)-z).
\end{equation}
The TCE for system \eqref{example2} becomes
$\lambda=0.1(1+{\rm i})+L({\rm e}^{-\lambda}-1)$.
The SCCs, obtained by taking $\lambda={\rm i}\beta$ into the equation, are
$$
L_k(\beta)=\dfrac{{\rm i}\beta-0.1(1+{\rm i})}{{\rm e}^{-{\rm i}\beta}-1}, ~\beta\in(2k\pi,(2k+2)\pi), ~k\in\mathbb{Z}.
$$
As shown in Fig.~\ref{fig5}, the SCCs separate complex plane into a few number of regions.  Clearly, ${\rm NU}(0)=1$. According to Theorem \ref{theorem1}, we have ${\rm NU}({\rm Region}~A)=1$. Using Theorem \ref{geometric}, we obtain that ${\rm NU}({\rm Region}~B)={\rm NU}({\rm Region}~D)={\rm NU}({\rm Region}~C)=2$, ${\rm NU}({\rm Region}~E)=3$, and ${\rm NU}({\rm Region}~G)=0$. The stability region $\Omega$ for the complex-valued $L$ is Region $G$, which is highlighted by green in  Fig.~\ref{fig5}.\hfill{$\square$}\end{example}

\subsection{Parametric representation of the polar coordinates for SCCs}\label{sec5}
In this subsection, we investigate the parametric representation of the polar coordinates for the SCCs, which is beneficial for identifying the critical parameter values at which the shape of the SCCs undergoes. 

\begin{figure}
\begin{center}
\centering\includegraphics[width=0.45\textwidth]{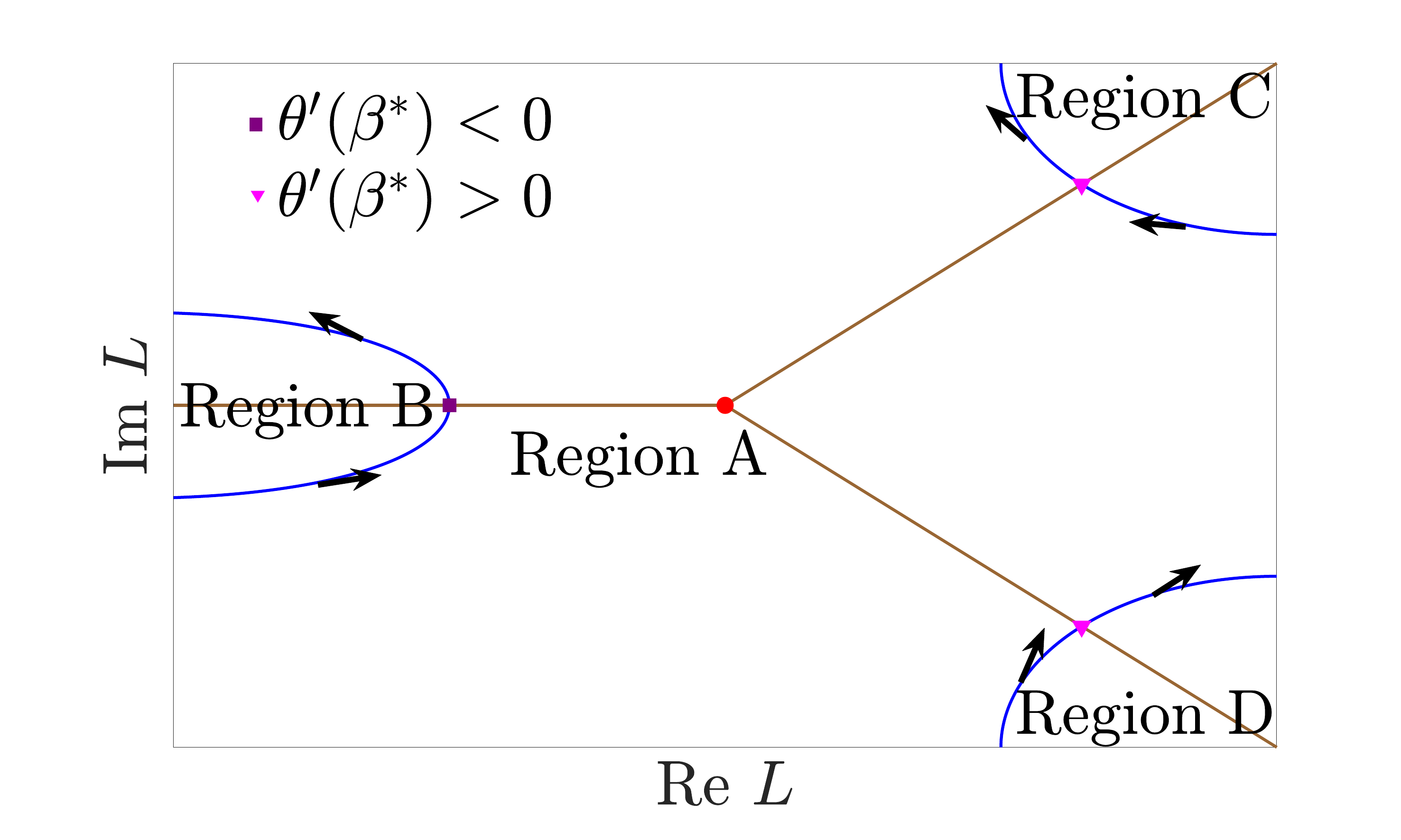}
\caption{The directions of the SCCs (the blue curves) are determined by different signs of $\theta'(\beta)$.  Here,  the black arrows represent the increasing directions of $\beta$. Multiple brown lines, originating from the origin (the red dot), intersect the SCCs at various points. The purple square dot corresponds to $\theta'(\beta^*)<0$, while the magenta triangle dots correspond to $\theta'(\beta^*)>0$. From Theorem \ref{theorem3}, ${\rm NU}({\rm Region}~B)-{\rm NU}({\rm Region}~A)=-1$, 
${\rm NU}({\rm Region}~C)-{\rm NU}({\rm Region}~A)=1$, and 
${\rm NU}({\rm Region}~D)-{\rm NU}({\rm Region}~A)=1$.} \label{fig25}
\end{center}
\end{figure}

Suppose that the SCCs are locally parameterized as $L(\beta)=r(\beta){\rm e}^{{\rm i}\theta(\beta)}$. Let us consider a ray $\gamma_{\theta}(t)\triangleq t{\rm e}^{{\rm i}\theta}$ with $t\in[0,+\infty)$ and $\theta\in [0,2\pi)$, starting from the origin and intersecting the SCCs at certain points.  We obtain the following theorem (see Fig. \ref{fig25}). 
\begin{theorem}\label{theorem3}
 Suppose that the ray $\gamma_{\theta^*}$ intersects the SCCs at $L^*=L(\beta^*)=\gamma_{\theta^*}(t^*)$, and that it is locally parameterized as $L=L(\beta)=r(\beta){\rm e}^{{\rm i}\theta(\beta)}$ with $\beta$ in the vicinity of $\beta^*$. Then, ${\rm NU}(\gamma_{\theta^*}(t^*+\epsilon))-{\rm NU}(\gamma_{\theta^*}(t^*-\epsilon))={\rm Sgn}(\theta'(\beta^*))$ for $\theta'(\beta^*)\neq 0$.
\end{theorem}
\begin{proof} By substituting  $L=\gamma_{\theta^*}(t^*\pm\epsilon)$ into Eq.~\eqref{yinhanshu0},  we have
$$
\lambda(\gamma_{\theta^*}(t\pm\epsilon))-{\rm i}\beta^*=\pm\epsilon\lambda'(L^*){\rm e}^{{\rm i}\theta^*}+o(\epsilon).
$$
Moreover, we have
$$
\lambda'(L^*)=-\dfrac{\partial_{L}F(\lambda^*,L^* )}{\partial_{\lambda}F(\lambda^*,L^*)}=\dfrac{{\rm i}}{L'(\beta^*)}.
$$
The second equality is derived from $F({\rm i}\beta, L(\beta))=0$. Differentiating with respect to $\beta$ at $\beta=\beta^*$ for $L(\beta)=r(\beta){\rm e}^{{\rm i}\theta(\beta)}$, we obtain that
$$
L'(\beta^*)={\rm e}^{{\rm i}\theta^*}[r'(\beta^*)+{\rm i}r(\beta^*)\theta'(\beta^*)].
$$
Thus, we have
$$\begin{aligned}
\lambda'(L^*){\rm e}^{{\rm i}\theta^*}&=\dfrac{{\rm i}}{L'(\beta^*)}{\rm e}^{{\rm i}\theta^*}=\dfrac{{\rm i}}{{\rm e}^{{\rm i}\theta^*}[r'(\beta^*)+{\rm i}r(\beta^*)\theta'(\beta^*)]}{\rm e}^{{\rm i}\theta^*}\\
&=\dfrac{{\rm i}}{r'(\beta^*)+{\rm i}r(\beta^*)\theta'(\beta^*)}=\dfrac{r(\beta^*)\theta'(\beta^*)+{\rm i} r'(\beta^*)   }{  r'(\beta^*)^2+  r(\beta^*)^2\theta'(\beta^*)^2},\\
\end{aligned}$$
which further implies that ${\rm Sgn}~[\lambda(\gamma_{\theta^*}(t^*\pm\epsilon))]=\pm{\rm Sgn}~[\theta'(\beta^*)]$. This indicates that, \textit{as $L$ moves from $\gamma_{\theta^*}(t^*-\epsilon)$ to $\gamma_{\theta^*}(t^*+\epsilon)$, one of the roots of the TCE $F(\lambda,L)=0$ undergoes a continuous transition from $\mathbb{C}_-$ to $\mathbb{C}_+$ when $\theta'(\beta^*)>0$  (or from $\mathbb{C}_+$ to $\mathbb{C}_-$ when $\theta'(\beta^*)<0$), crossing through $\mathbb{C}_0$ at ${\rm i}\beta^*$ when $L=L^*$} (see Fig. \ref{fig25}). Consequently, this implies that ${\rm NU}(\gamma_{\theta^*}(t^*+\epsilon))-{\rm NU}(\gamma_{\theta^*}(t^*-\epsilon))={\rm Sgn}(\theta'(\beta^*))$.
\end{proof}


\section{Applications to scalar DDEs}\label{sec5.3}

In this section, we employ the geometric approach to determine the stability region $\Omega$ of the scalar system $\dot{z}=az+L\int_0^{+\infty}h(\tau)z(t-\tau){\rm d}\tau$ for different scenarios of $h(\tau)$. These analyses have practical implications for real-world systems, as detailed in Examples~\ref{carfollowing3}-\ref{qutongbu}. First, we consider the scenario in which $h(\tau)$ is a Dirac delta function.

\begin{figure}
\begin{center}
\includegraphics[width=0.45\textwidth]{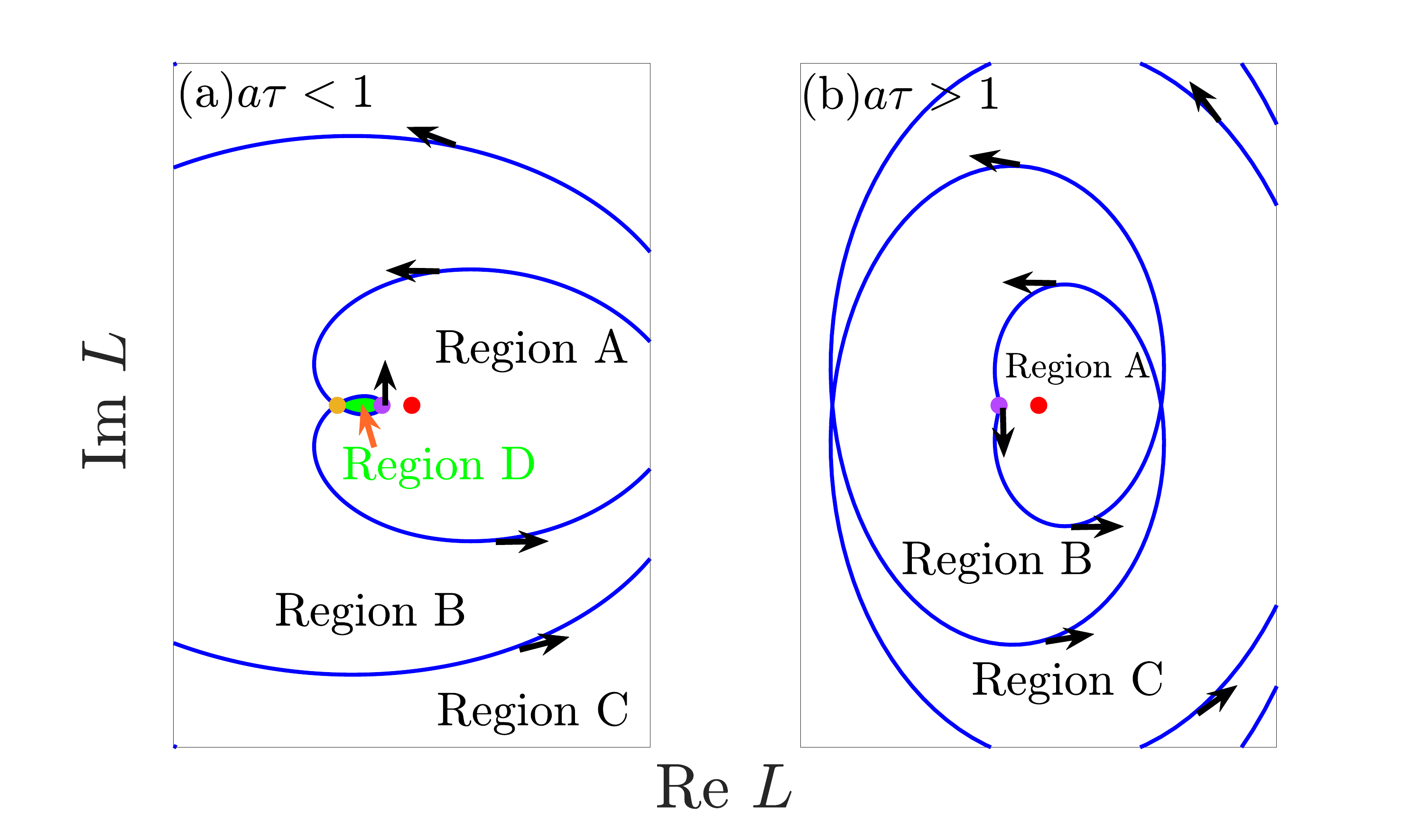}\label{fig13}
\subfigure{\label{fig13a}} \subfigure{\label{fig13b}}
\caption{The SCCs for system \eqref{example3} have different directions at $\beta=d$ (purple dots) for  $a\tau<1$ (a) and $a\tau>1$  (b). The black arrow indicates the increasing direction of the SCCs with respect to $\beta$. Here, the red dots represent the origin, which implies ${\rm NU}({\rm Region}~A)=1$. According to Theorem \ref{geometric}, 
${\rm NU}({\rm Region}~B)=2$, 
${\rm NU}({\rm Region}~C)=3$, and
${\rm NU}({\rm Region}~D)=0$. Thus, the stability regions $\Omega_{a,d,\tau}$ are the region colored by green in (a). 
 The brown dot in (a) represents the inflection point $L(d\pm\beta^*)$, where $\beta^*>0$ denotes the smallest positive root for which $\beta^*\tau-\arctan({\beta^*}/{a})=0$. The existence of the stability region depends on the direction of the SCCs at  $\beta=d$, which is determined by the sign of ${\theta}'(d)$.} \label{fig13}
\end{center}
\end{figure}

\begin{figure}
\begin{center}
\centering\includegraphics[width=0.45\textwidth]{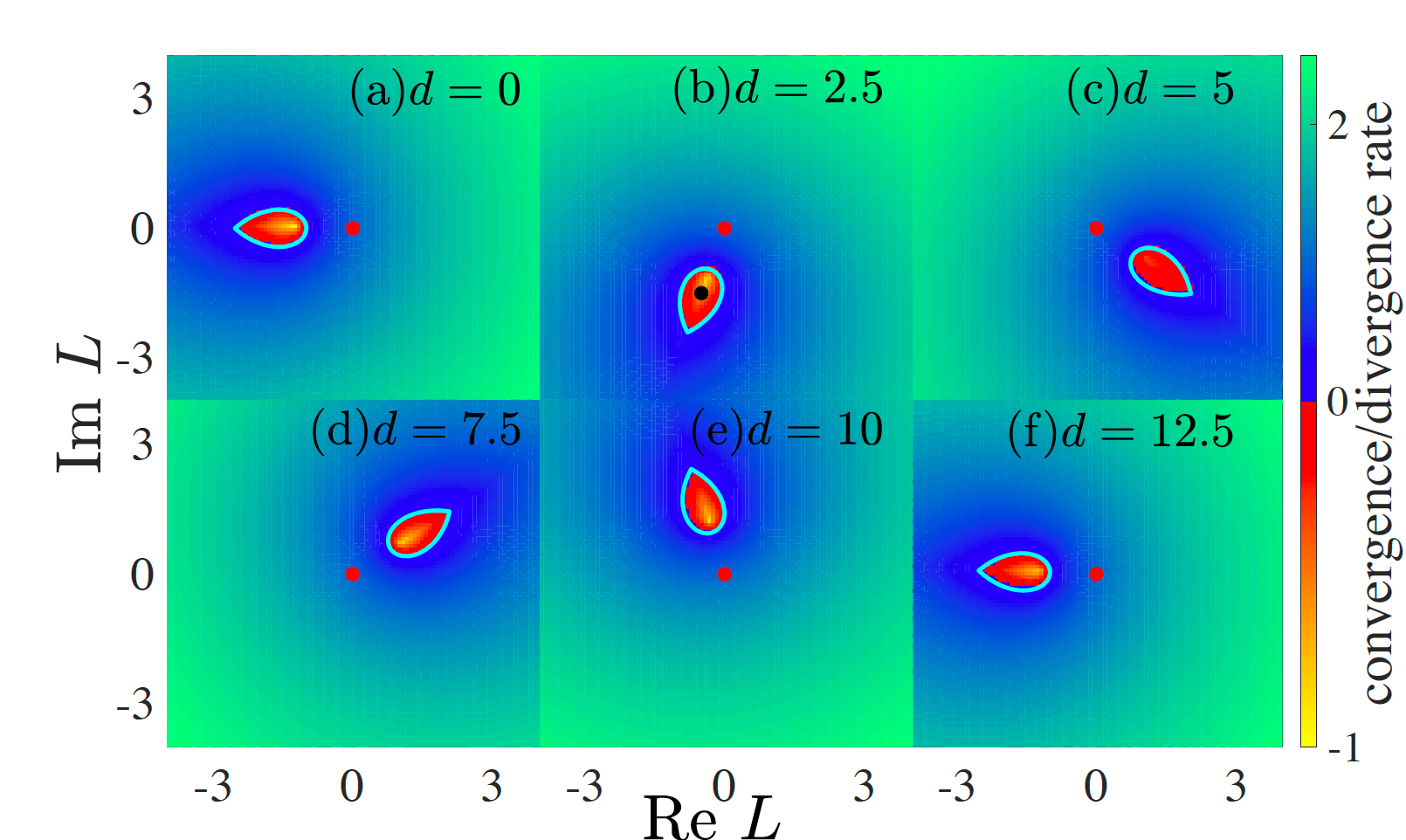}
\subfigure{\label{fig12a}} \subfigure{\label{fig12b}}
\caption{The stability region $\Omega_{a,d,\tau}$ for system \eqref{example3} with $a\tau<1$ appears rotationally around the origin (highlighted by the red dots) in the complex plane of $L$ for different $d$, including $d=0$ (a), $d=2.5$ (b), $d=5$ (c), $d=7.5$ (d), $d=10$ (e), and $ d=12.5$ (f). Here, the contours of the leaf-shaped stability regions are depicted by the solid cyan curves according to \eqref{stabilityregion} with $a=1$ and $\tau=0.5$. The rotation period is $2\pi/\tau\approx 12.56$.  The colors represent the exponential  convergence or divergence rates of the trajectories numerically generated by system \eqref{example3}, which are calculated by $\limsup_{t\to+\infty}\log|z_t|/t$. The black dot in (b) corresponds to the choice of the parameters for $C$ and $S$ in Fig. \ref{fig10b}.} \label{fig12}
\end{center}
\end{figure}

\begin{example}\label{example5}
Consider a system with a discrete delay as
\begin{equation}\label{example3}
\dot{z}=(a+{\rm i}d)z+Lz(t-\tau),
\end{equation}
where $a>0$, $d\in\mathbb{R}$, and $\tau>0$ is the time delay.  For this particular scenario, the stability regions, denoted by $\Omega_{a,d,\tau}$, have been completely investigated in \cite{b10,b11}.  However, to test the efficacy of the developed approach, we still investigate this problem.  Correspondingly, the TCE is given by
$$F(\lambda,\tau)\triangleq\lambda-a-{\rm i}d-L{\rm e}^{-\lambda\tau}=0.$$ The SCCs, obtained by taking $\lambda={\rm i}\beta$ into the equation, are
$
L=L(\beta)=-{\rm e}^{{\rm i}\beta\tau}[a-{\rm i}(\beta-d)]=r(\beta){\rm e}^{{\rm i}\theta(\beta)},
$
where
$$
r(\beta)=\sqrt{a^2+(\beta-d)^2}, ~~\theta(\beta)=\pi+\beta\tau-\arctan\dfrac{(\beta-d)}{a}.
$$
It follows that
$\theta'(\beta)=\tau-{a}/{[a^2+(\beta-d)^2]}$,
where $\theta'(\beta)$ achieves its minimum $\tau-{1}/{a}$ at $\beta=d$.  As seen in Fig. \ref{fig13}, the SCCs have different directions at $\beta=d$ for $a\tau<1$ and $a\tau>1$. 

When $a\tau<1$, as seen in Fig. \ref{fig13a},  the SCCs undergo self-intersection at $\beta=d\pm\beta^*$, where $\beta^*>0$ denotes the smallest positive root such that $\beta^*\tau-\arctan({\beta^*}/{a})=0$. This segment of the curve, $L=L(\beta),\beta\in[d-\beta^*,d+\beta^*]$, forms a stability region (depicted by a green shaded region in Fig. \ref{fig13a}). Here, $L(d-\beta^*)=L(d+\beta^*)$ represents the  inflection point of the SCCs (highlighted by the brown dot in Fig. \ref{fig13a}). By performing a parameter transformation $\hat{\beta}=\beta-d$, the contours of the stability regions are expressed as
\begin{equation}\label{stabilityregion}
L=L_d(\hat{\beta})=-{\rm e}^{{\rm i}d\tau}{\rm e}^{{\rm i}\hat{\beta}\tau}(a-{\rm i}\hat{\beta}), ~~\hat{\beta}\in[-\beta^*,\beta^*].
\end{equation}
From \eqref{stabilityregion}, it is inferred that, when the parameters $\tau$ and $a$ are fixed, the stability regions $\Omega_{a,d,\tau}$ appear rotationally around the origin as $d$ increases, with a rotation period of ${2\pi}/{d}$.  Additionally, Figure~\ref{fig12} shows these counterclockwise rotating and leaf-shaped regions, along with the consistent numerical results obtained from system \eqref{example3}.

When $a\tau>1$, as seen in Fig. \ref{fig13b}, there is no stability region in the whole complex plane. \hfill{$\square$}
\end{example}

Next, we consider the scenario in which $h(\tau)$ follows a Gamma distribution.

\begin{example}\label{examplelizi2}
Consider a system with distributed delay as
\begin{equation}\label{example4}
   \dot{z}=az+L\int_0^{+\infty}z(t-\tau)h^n_T(\tau){\rm d}\tau,
\end{equation}
where  $a>0$. 
We suppose that the distribution of the delay obeys Gamma's distribution 
\begin{equation}\label{gammadistribution}
h^n_T(\tau)\triangleq\dfrac{n^n}{(n-1)! T^n}\tau^{n-1}{\rm e}^{-\frac{\tau n}{T}},
\end{equation}
where parameter $n\in N^*$ and $T>0$ represents the mean value of the distribution. 
In the following, we investigate the stability region $\Omega_T^n$ for different values of $n$ and $T$.  Correspondingly, the TCE for system \eqref{example4} becomes
$\lambda=a+L \int_0^{+\infty} {\rm e}^{-\lambda\tau}h^n_T(\tau){\rm d}\tau$, where
$$
\int_0^{+\infty} {\rm e}^{-\lambda\tau}h^n_T(\tau){\rm d}\tau=\left(1+\dfrac{\lambda T}{n}\right)^{-n}.
$$
The SCCs, obtained by taking $\lambda={\rm i}\beta$ into the equation, are
\begin{equation}\label{stabilityregion2}
L=L^n_T(\beta)=-(a-{\rm i}\beta)\left(1+\dfrac{{\rm i}\beta T}{n}\right)^n.
\end{equation}
Thus, the SCCs can be written as  $L^n_T(\beta)=r^n_T(\beta){\rm e}^{{\rm i}\theta^n_T(\beta)}$, where
 $$
 r^n_T(\beta)=(a^2+\beta^2)^{\frac{1}{2}}\left(1+\dfrac{\beta^2T^2}{n^2}\right)^{\frac{n}{2}}
 $$
 and
 $$
 \theta^n_T(\beta)=\pi+n\arctan\left(\dfrac{\beta T}{n}\right)-\arctan\left(\dfrac{\beta}{a}\right).
 $$
 Consequently, we obtain ${\theta^n_T}'(0)=T-{1}/{a}$.

Case I: $n=1$. As shown in Figs.~\ref{fig7a}-\ref{fig7b}, the
SCCs separate the complex plane into two regions. They have different directions dependent on different groups of parameters. The critical value is determined by ${\theta^n_T}'(0)=0$. Clearly, ${\rm NU}(0)=1$, and thus ${\rm NU}({\rm Region}~A)=1$. According to Theorem \ref{theorem3}, ${\rm NU}({\rm Region}~B)-{\rm NU}({\rm Region}~A)={\rm Sgn}[{\theta^n_T}'(0)]$. As a result, for $aT<1$, ${\rm NU}({\rm Region}~B)=0$, which implies that the stability set $\Omega_T^n$ for the complex-valued $L$ is Region $B$, the green shaded region shown in Fig. \ref{fig7a}.  In addition, Figure~\ref{fig14} shows the unbounded stability region, along with the consistent numerical results obtained from system \eqref{example4}.   Conversely for $aT>1$, ${\rm NU}({\rm Region~B})=2$, so that there is no stability region for $L$ (see Fig. \ref{fig7b}).

Case II: $n\geq 2$.  As shown in Fig. \ref{fig7c}, for $aT<1$, the SCCs separate the complex plane into three regions. The SCCs through self-encirclement result in a formation of an additional region, referred to as Region $C$ which is the stability region $\Omega_T^n$ (highlighted by the green shaded area  in Fig. \ref{fig7c}). The contours of this region are expressed as $L=L_T^n(\beta)$  for $\beta\in [-\beta^*,\beta^*]$. Here, $\beta^*>0$ denotes the smallest positive root satisfying $n\arctan({\beta^* T}/{n})-\arctan({\beta^*}/{a})=0$. 
Moreover, as shown in Fig. \ref{fig7d}, for $aT>1$, the SCCs separate the complex plane into two regions.  Analogous to the situation for $n=1$, there is no stability region. 
\hfill{$\square$}
\end{example}

\begin{figure}
\begin{center}
\centering\includegraphics[width=0.45\textwidth]{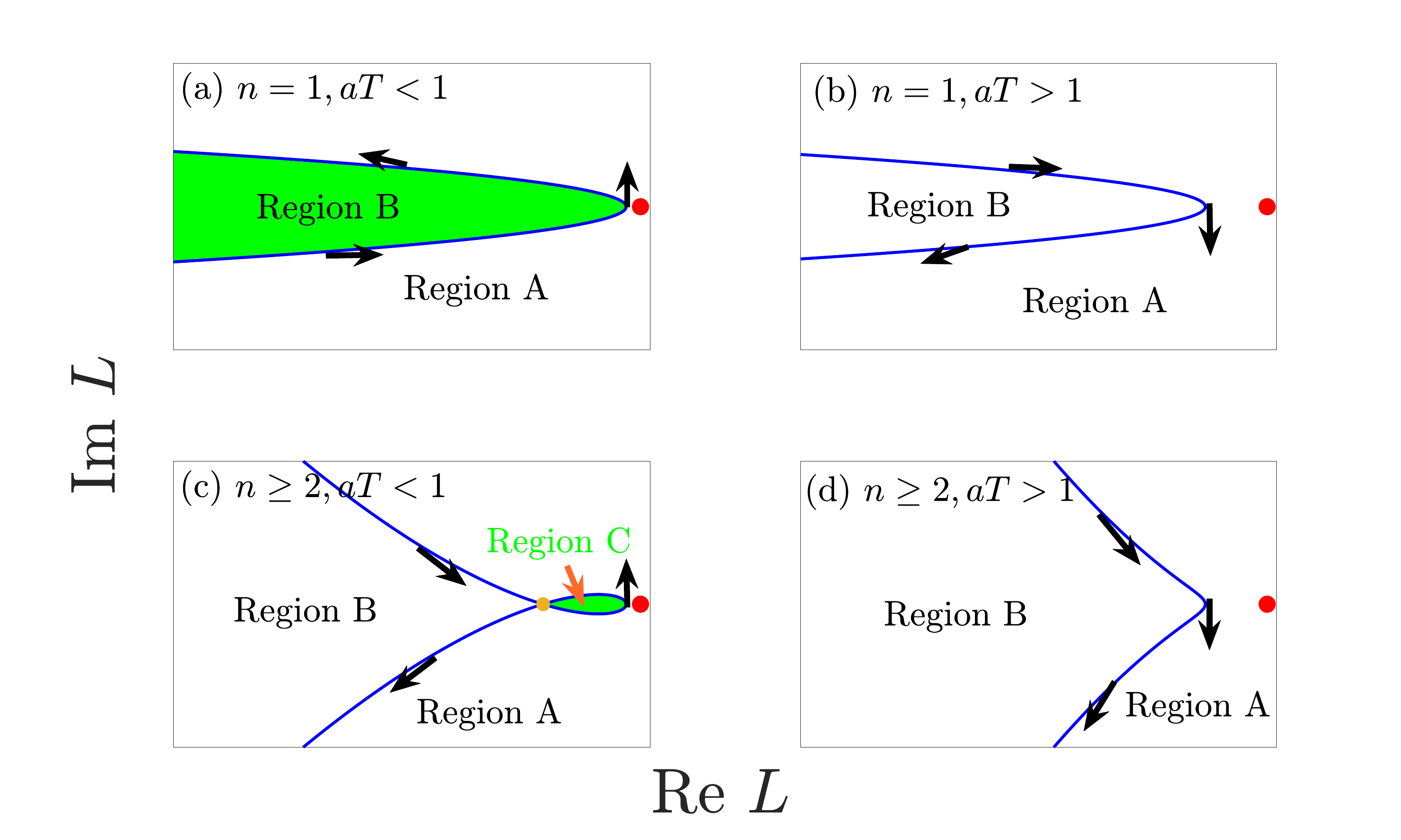}
\subfigure{\label{fig7a}} \subfigure{\label{fig7b}} \subfigure{\label{fig7c}} \subfigure{\label{fig7d}}
\caption{The SCCs, determined by \eqref{stabilityregion2}, have different shapes and directions for different values of $a$, $T$, and $n$. The parameters are, respectively, chosen as: $n=1$ and $aT<1$ (a), $n=1$ and $aT>1$  (b), $n\geq 2$ and $aT<1$ (c), and $n\geq 2$ and $aT>1$ (d).  The  black arrow indicates the increasing direction of the SCCs with respect to $\beta$. Here, the red dots represent the origin, so that  ${\rm NU}({\rm Region}~A)=1$. The value of ${\rm NU}$ for each region is obtained by using Theorem \ref{geometric}, which gives the stability region $\Omega_T^n$ for system \eqref{example4} (see the green shaded region).  The brown dot in (c) represents the inflection point $L^n_T(\beta^*)$. The existence of the stability region depends on the direction of the  SCCs at  $\beta=0$, which is determined by the sign of ${\theta^n_T}'(0)$.} \label{fig7}
\end{center}
\end{figure}

\begin{figure}
\begin{center}
\centering\includegraphics[width=0.45\textwidth]{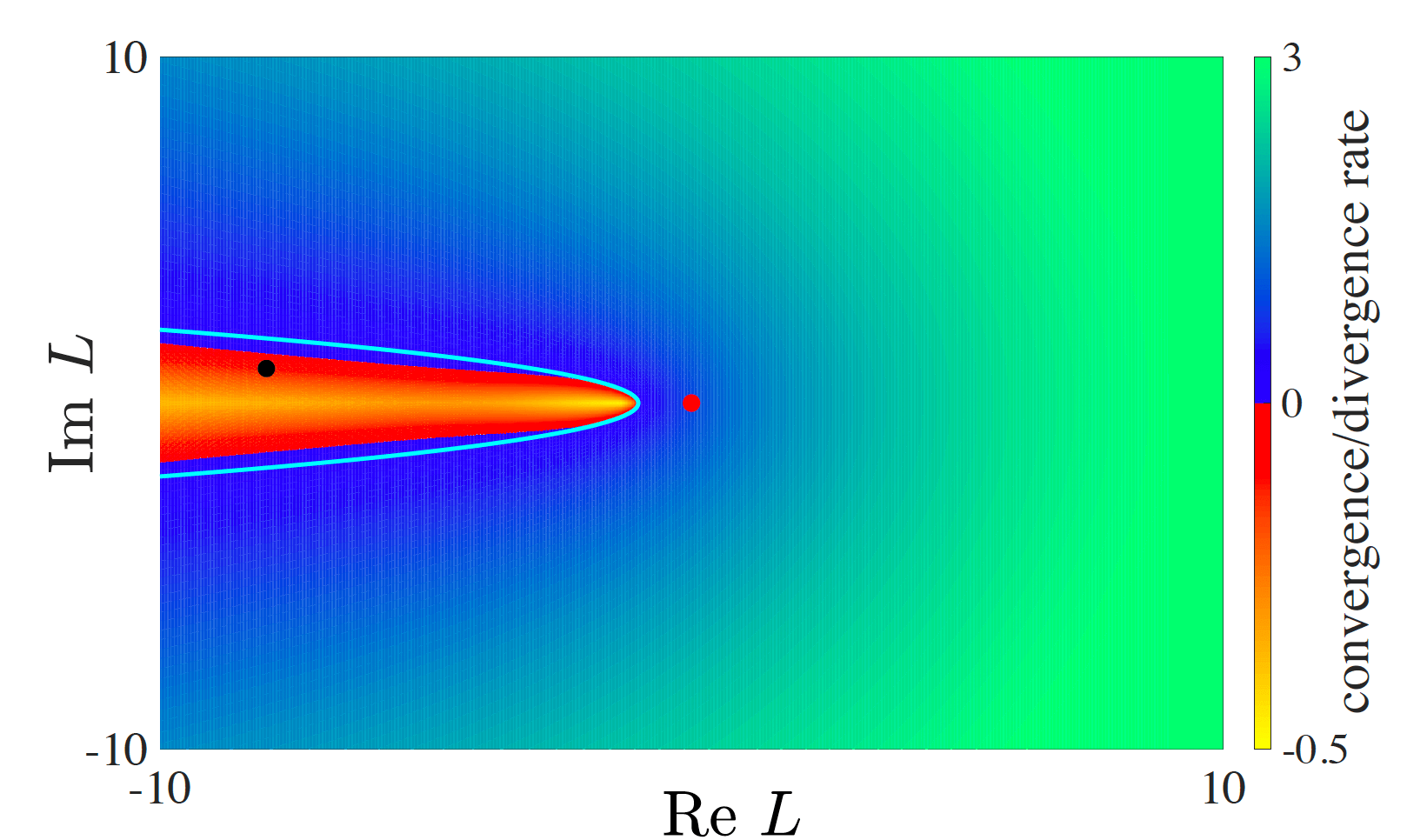}
\caption{The unbounded stability region appears in the left half of the complex plane of $L$ for system \eqref{example4} with $aT<1, n=1$. Here, the contours of the stability regions are depicted by the solid cyan curves according to \eqref{stabilityregion2} with $n=1$, $a=1$, $T=0.5$. The colors, calculated by $\limsup_{t\to+\infty}\log|z_t|/t$, represent the exponential convergence or divergence rates of the trajectories numerically generated by system \eqref{example4}. The black dot  corresponds to the choice of the parameters for $C$ and $S$ in Fig. \ref{fig10a}.} \label{fig14}
\end{center}
\end{figure}

\section{Applications to networked systems}\label{sec5.5}

In this section, we provide several examples of real-world systems, enhancing the wide-ranging applicability of the proposed approach in various domains, such as transportation, engineering, and biomedical engineering.

Both Examples \ref{carfollowing3} and \ref{MAS0} demonstrate the practical application of the proposed approach in analyzing the consensus/stability problem of MASs.   In particular, Example \ref{carfollowing3} investigates the car-following system, providing valuable insights into the complex dynamics of traffic.

\begin{figure}
\begin{center}
\centering\includegraphics[width=0.45\textwidth]{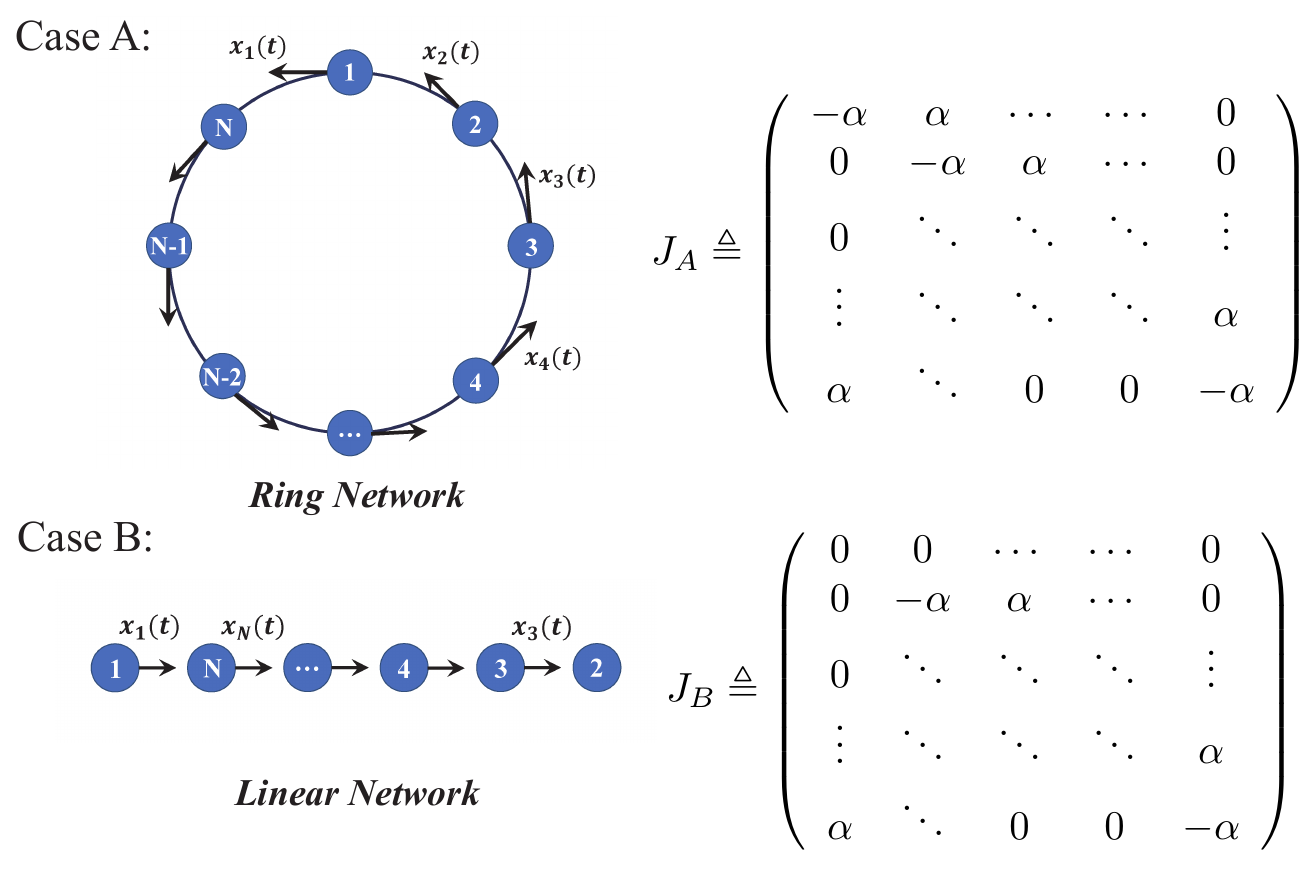}
\caption{The ring (Case A) and the linear (Case B) networks for $N$ vehicles corresponding to the asymmetric matrices $\bm{J}_A$ and $\bm{J}_B$, respectively.}\label{fig24}
\end{center}
\end{figure}

\begin{figure*}
\begin{center}
\centering
\subfigure[$n=1$]{\includegraphics[width=0.4\textwidth]{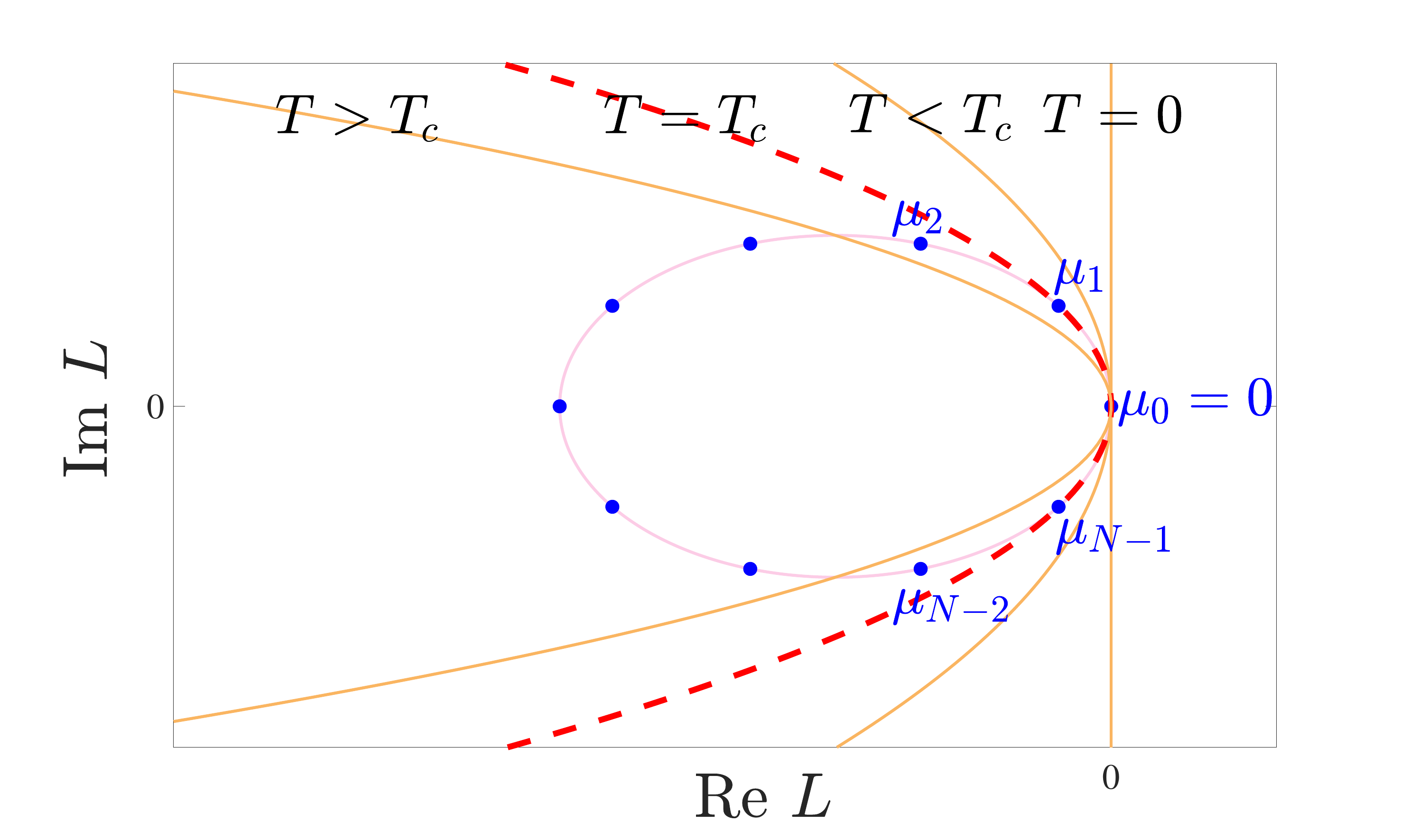}\label{fig15a}}
\subfigure[$n\geq 2$]{\includegraphics[width=0.4\textwidth]{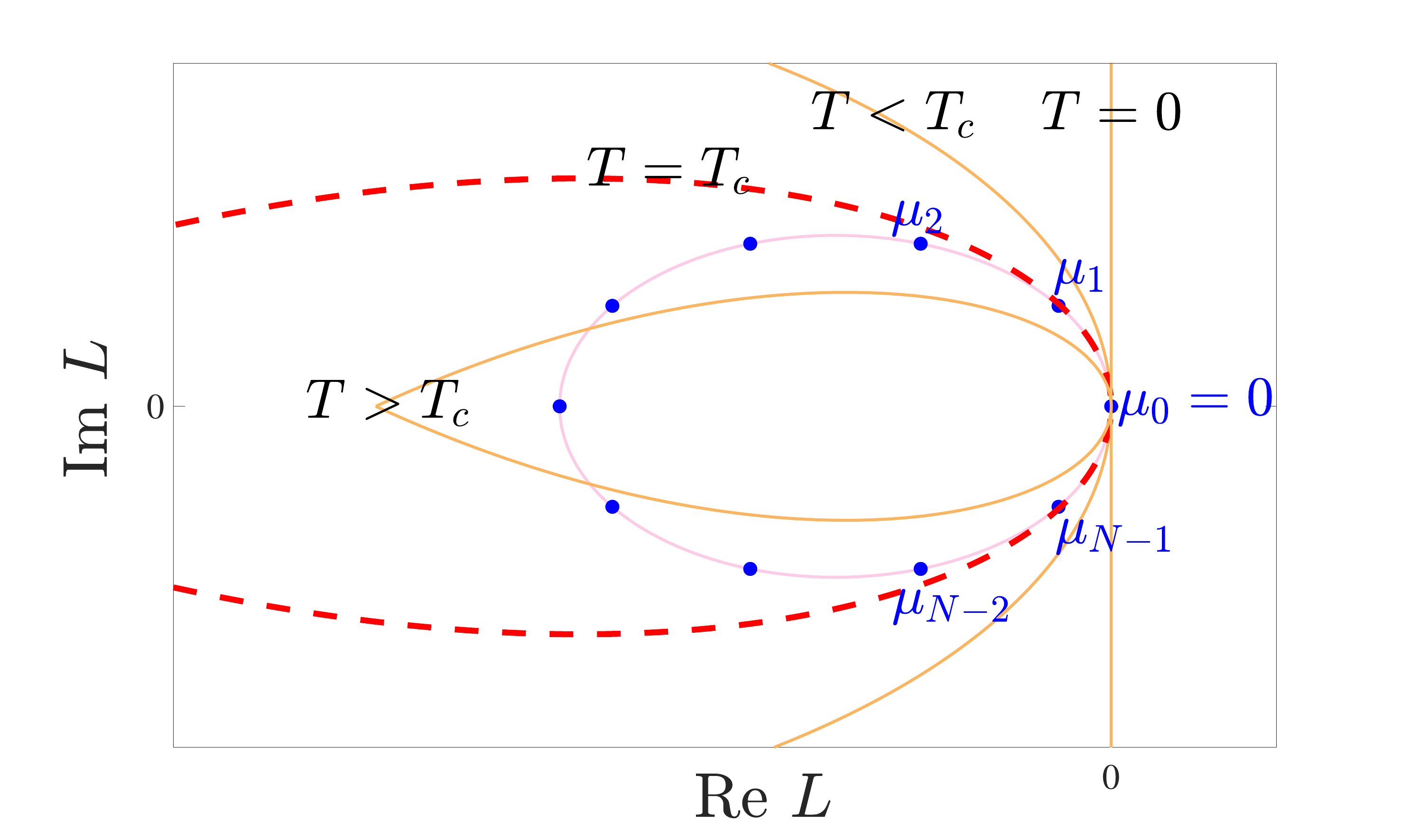}\label{fig15b}}
\caption{The stability region $\Omega_T^n$ for system \eqref{tezhengfangcheng}, circumscribed by the contours (the brown curves) determined by \eqref{stabilityregion3}, shrinks as $T$ increases. For $T=0$, the stability region  reduces to the left half of the complex plane.  For $T>0$, (a) $\Omega_T^n$ remains unbounded as $n=1$, and (b) $\Omega_T^n$ becomes a bounded leaf-shaped region as $n\geq 2$. In Case A, all eigenvalues of the asymmetric network $\bm{J}_{A}$, highlighted by the blue dots, are evenly distributed on the circle obeying the equation $|\mu+\alpha|=\alpha$ (represented by the pink curves) in the complex plane. As $T$ increases from $0$, all eigenvalues (except for $0$) are located within the stability region until it surpasses the critical value $T_c$. The value of $T_c$ is determined by the circumstances in which the conjugate pair of the eigenvalues $\{\mu_1,  \mu_{N-1}\}$ intersect with the contours of $\Omega_T^n$ (see the red dashed curves).} \label{fig15}
\end{center}
\end{figure*}

\begin{example}\label{carfollowing3}
We investigate the consensus problem of a class of car-following systems with distributed delays, which reads
\begin{equation}\label{carfollowing2}
\dot{x}_i=\alpha_i \int_0^{+\infty} h(\tau)\big[x_{i+1}(t-\tau)-x_i(t-\tau)\big]{\rm d}\tau,
\end{equation}
with $i=1,\cdots,N$ and $x_{N+1}=x_1$. Here, $x_i(t)$ is the velocity of the $i$-th vehicle at time $t$, $\alpha_i>0$ is seen as the sensitivity of the $i$-th driver to the velocity difference between the current vehicle and the preceding vehicle, and $h(\tau)$ is the time delay distribution. Using the master stability function (see details in Section \ref{MSF}), we obtain the necessary and sufficient condition on the consensus of system \eqref{carfollowing2} (i.e. $\lim_{t\to+\infty}|{x}_i(t)-{x}_j(t)|=0$ for $1\leq i<j\leq N$): All eigenvalues (except for $0$) of the network matrix $\bm{J}$, which is weighted by $\alpha_i\geq 0$, are located within the stability region $\Omega$ of 
\begin{equation}\label{tezhengfangcheng}
\dot{z}=L \int_0^{+\infty}z(t-\tau)h(\tau){\rm d}\tau.
\end{equation}
In the following, we discuss about two representative cases for the network matrix of $N$ vehicles (see Fig. \ref{fig24}).  Case A: vehicles traveling around a ring network (the asymmetric matrix $\bm{J}_A$  corresponding to the situation where $\alpha_i\equiv \alpha$ in system \eqref{carfollowing2}), and Case B: vehicles arranged along a linear network (the asymmetric matrix $\bm{J}_B$ corresponding to the situation in which $\alpha_1=0$ and $\alpha_i\equiv\alpha$ for $2\leq i\leq N$ in system \eqref{carfollowing2}).  The distribution density is taken as $h(\tau)=h^n_T(\tau)$, which has been defined in \eqref{gammadistribution}.   

Using the arguments similar to those in Example \ref{examplelizi2}, we obtain the stability regions $\Omega^n_T$ for system \eqref{tezhengfangcheng} with $h(\tau)=h^n_T(\tau)$  within the contours that are determined by
\begin{equation}\label{stabilityregion3}
    L=L^n_T(\beta)={\rm i}\beta\left(1+\dfrac{{\rm i}\beta T}{n}\right)^n, ~\beta\in(-\beta^*,\beta^*), ~\beta^*\triangleq\dfrac{n}{T}\tan\left(\dfrac{\pi}{2n}\right).
\end{equation}
Here, we regard $\tan({\pi}/{2})=+\infty$. As seen in Fig. \ref{fig15}, the stability region shrinks as $T$ increases from $0$, i.e., $\Omega^n_{T_1}\subseteq\Omega^n_{T_2}$ for $T_1>T_2$.  As $T=0$, the stability region reduces to the left half of the complex plane. 
 
 In Case A, all eigenvalues of $\bm{J}_A$ are given by 
 \begin{equation}\label{335}
  \mu_l=\alpha\left({\rm e}^{{\rm i}\frac{2\pi l}{N}}-1\right)=2\alpha\sin\left(\dfrac{\pi l}{N}\right){\rm e}^{{\rm i}\left(\frac{\pi}{2}+\frac{\pi l}{N}\right)}, ~ l=0,1,\cdots,N-1.
 \end{equation}
As $T$ increases from $0$, all eigenvalues (except for $0$) are located within $\Omega^n_{T}$  until reaching a critical value $T_c$, where one of the eigenvalues (except for $0$) touches the contours of $\Omega^n_{T}$ (see the red dashed curves in Fig. \ref{fig15}). Now, we turn to calculate the critical value $T_c$, which is written as
\begin{equation}\label{linjiezhi}
 \mu_l=L^n_{T_c}(\beta_c),~ l\neq 0.
\end{equation}
Write the contour \eqref{stabilityregion3} using the  polar coordinates as $L^n_T(\beta)=r^n_T(\beta){\rm e}^{{\rm i}\theta^n_T(\beta)}$, where
$$
 r^n_T(\beta)=|\beta|\left(1+\dfrac{\beta^2T^2}{n^2}\right)^{\frac{n}{2}},  ~~ 
 {\theta}^n_T(\beta)=\dfrac{\pi}{2}+n\arctan\left(\dfrac{\beta T}{n}\right).
  $$
Then, taking the absolute value and argument value separately on both sides of \eqref{linjiezhi}, we obtain that
\begin{equation}\label{111}
\beta_c\left(1+\dfrac{\beta_c^2T_c^2}{n^2}\right)^{\frac{n}{2}}=2\alpha\sin\left(\dfrac{\pi l}{N}\right), ~\beta_c>0, ~ 0<k\leq\dfrac{N}{2}
\end{equation}
and
\begin{equation}\label{222}
n\arctan\left(\dfrac{\beta_cT_c}{n}\right)=\dfrac{\pi l}{N}.
\end{equation}
It is important to note that the contours specified in \eqref{stabilityregion3} as well as all eigenvalues are symmetric along the real axis. Therefore, we select the critical value situated above the real axis, resulting in $\beta_c>0$ while $0<l\leq{N}/{2}$. It follows from \eqref{222} that ${\beta_c T_c}/{n}=\tan\left({\pi l}/{Nn}\right)$. Substituting it into \eqref{111} yields:
$$
\beta_c=\dfrac{2\alpha\sin\left(\dfrac{\pi l}{N}\right)}{\left[1+\tan^2\left(\dfrac{\pi l}{Nn}\right) \right]^\frac{n}{2}}, ~~T_c=\dfrac{n\tan\left(\dfrac{\pi l}{Nn}\right)\left[1+\tan^2\left(\dfrac{\pi l}{Nn}\right) \right]^\frac{n}{2}}{2\alpha\sin\left(\dfrac{\pi l}{N}\right)}.$$
 Taking $l=1$ further yields the sufficient and necessary condition on the consensus of system \eqref{carfollowing2} for Case A as:
\begin{equation}\label{consensus}
0<T<T_c=\dfrac{n\tan\left(\dfrac{\pi }{Nn}\right)\left[1+\tan^2\left(\dfrac{\pi }{Nn}\right) \right]^\frac{n}{2}}{2\alpha\sin\left(\dfrac{\pi }{N}\right)}.
\end{equation}
To numerically validate the analytically-obtained results, we choose four combinations for the pair $(n,N)$, viz., (1,10), (1,5), (2,10), and (2,5). For each combination,  system \eqref{carfollowing2} is numerically implemented with the parameters $(\alpha,T)\in(0,2)\times (0,2)$. The divergence or convergence synchronization rates are computed by $\limsup_{t\to+\infty}{\max_{1\leq i<j\leq N}\log|x_i(t)-x_j(t)|}/{t}$. As shown in Fig. \ref{fig17}, the analytical criteria in \eqref{consensus} are confirmed by our numerical results.

\begin{figure}
\begin{center}
\centering
\includegraphics[width=0.45\textwidth]{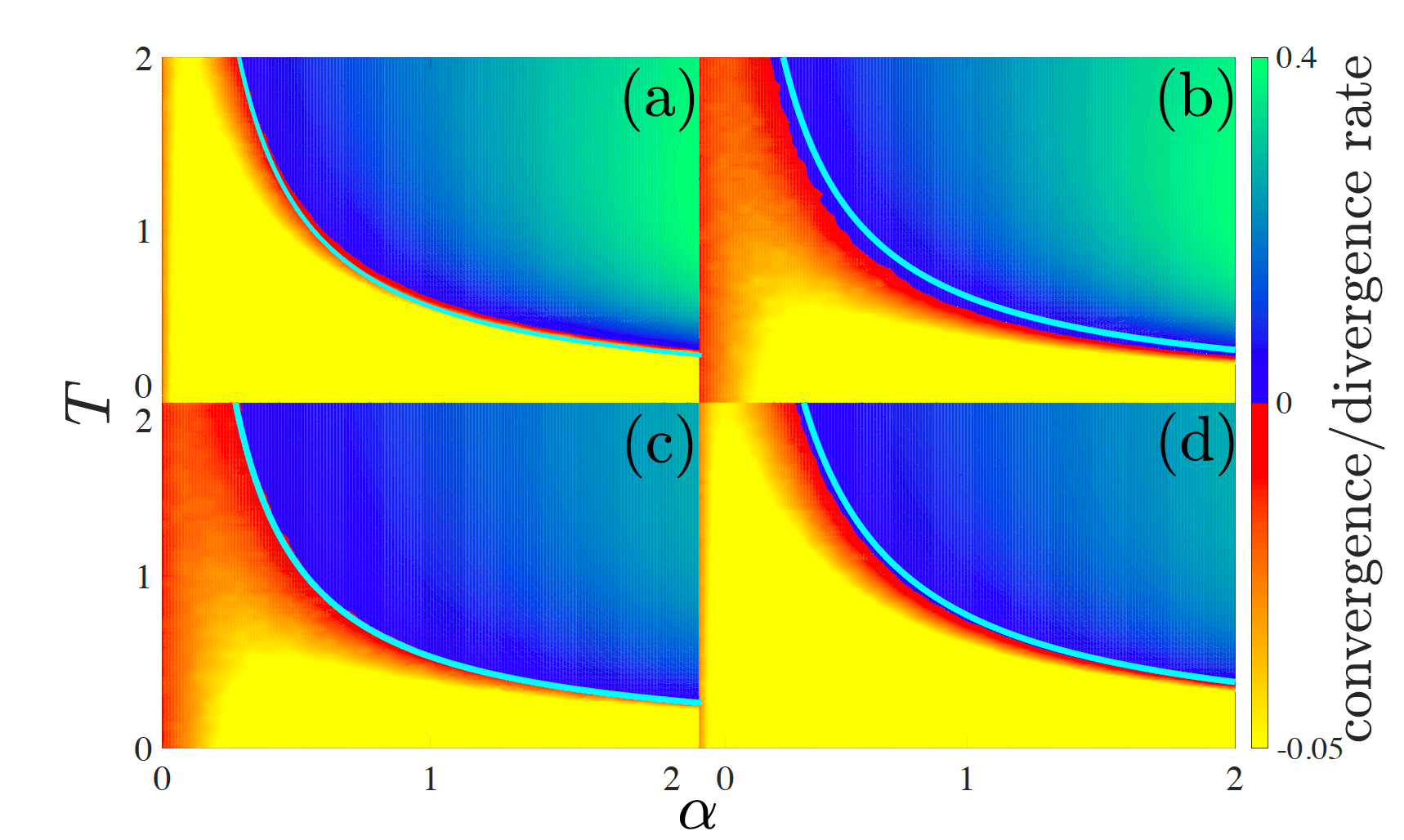}\label{fig17}
\caption{The stability regions with respect to $\alpha$ and $T$ for system \eqref{carfollowing2} with the asymmetric network $\bm{J}_A$, delay distribution $h^n_T(\tau)$ and different pairs of $(n,N)$. The contours of the stability region are depicted by the solid cyan curves according to \eqref{consensus}.  
The colors represent the divergence or convergence synchronization rates of the trajectories numerically generated by system \eqref{carfollowing2}, which are computed by $\limsup_{t\to+\infty}{\max_{1\leq i<j\leq N}\log|x_i(t)-x_j(t)|}/{t}$. Four combinations for the pair $(n,N)$ are used,  viz., $(1,10)$ (a), $(1,5)$ (b), $(2,10)$ (c), and $(2,5)$ (d).} \label{fig17}
\end{center}
\end{figure}

In Case B, all eigenvalues (except for $0$) of $\bm{J}_B$ are given by 
 \begin{equation}\label{336}
  \mu_l=-\alpha, ~l=1,\cdots,N-1.
 \end{equation}
 By \eqref{stabilityregion3}, we have $$\Omega_T^n\cap\mathbb{R}=\left(-\dfrac{n}{T}\tan\left(\dfrac{\pi}{2n}\right)\left[1+\tan^2\left(\dfrac{\pi }{2n}\right) \right]^\frac{n}{2},0\right).$$ Consequently, the sufficient and necessary condition on 
 the consensus of system \eqref{carfollowing2} for Case B becomes 
$$
0<T<T_c\triangleq\dfrac{n}{\alpha}\tan\left(\dfrac{\pi}{2n}\right)\left[1+\tan^2\left(\dfrac{\pi }{2n}\right) \right]^\frac{n}{2},
$$
 where we regard $\tan({\pi}/{2})=+\infty$. \hfill{$\square$}  \end{example} 


The next example investigates the application of proportional and derivative (PD) control with delays for controlling a second-order system coupled with \textit{random} networks. The second-order system is usually used to describe various physical phenomena in the context of mechanics and power grid, while the PD control is widely used in control engineering.

\begin{example}\label{MAS0}
Consider a general second-order MAS with $N$ agents, which reads
\begin{equation}\label{MAS}
\dot{x}_i(t)=v_i(t),~~\dot{v}_i(t)=av_i(t)+bx_i(t)+u_i(t),
\end{equation}
where $x_i(t)$, $v_i(t)$, and $u_i(t)$ denote position, velocity, and input of agent $i$, respectively, and $i=1,2,\cdots,N$. We introduce a PD control protocol with the distributed time delays for the input as
\begin{equation}\label{251}
\begin{aligned}
u_i(t)&=k_1\sum_{j=1}^N a_{ij}     \int_0^{+\infty}h(\tau)x_j(t-\tau){\rm d}\tau\\
&~~~ +k_2\sum_{j=1}^N a_{ij}     \int_0^{+\infty}h(\tau)v_j(t-\tau){\rm d}\tau.
\end{aligned}
\end{equation}
Here, $k_{1,2}$ represent the proportional and derivative gains, respectively, $a_{ij}$ denotes the adjacency weight of the connections for MAS network, and
$h(\tau)$ is the density distribution of time delays. Significant contributions have been made in prior research studies concerning the case involving discrete time delays (i.e., $h(\tau)=\delta(\tau-T)$)\cite{b47,b48}.  Here, we investigate the scenario involving the distributed time delays, where the density distribution is characterized by $h_T(\tau)=\frac{1}{T}{\rm e}^{-\frac{\tau}{T}}$.

Using similar argument to Section \ref{MSF}, we obtain the sufficient and necessary condition on the stability of system \eqref{MAS} with controller \eqref{251} (i.e. $\lim_{t\to+\infty} |{x}_i(t)|=\lim_{t\to+\infty} |v_i(t)|=0$ for all $i$) as:  All eigenvalues of $\bm{J}=\{a_{ij}\}_{N\times N}$ are located within the stability region $\Omega_T$ of the characteristic equation
\begin{equation}\label{MAS2}
\lambda^2-a\lambda-b-(k_1+k_2\lambda)L/(1+\lambda T)=0.
\end{equation}
It is noted here that the TCE is transformed as an equation of polynomial due to the particularly-used density distribution $h_T(\tau)$.   Thus, the SCCs becomes
\begin{equation}\label{MAS3}
   L=L_T(\beta)=\dfrac{(-\beta^2-b-{\rm i}a\beta)(1+{\rm i}\beta T)}{(k_1+{\rm i}k_2\beta)}.
\end{equation}
To this end, we are in a position to present the change of the stability regions using the geometric approach. As shown in Fig. \ref{fig18a}, the SCCs, as determined by the representation in \eqref{MAS3}, own different shapes and directions for different values of $T\geq 0$. As $T=0$, the SCCs divide the complex plane into two regions, where the left one defines the stability region $\Omega_0$ of Eq.~\eqref{MAS2}. As $T$ increases, the shapes and directions of the SCCs remain unchanged until $T$ surpasses a critical value $T_{c_1}$ (see Fig. \ref{fig18b}).  As $T$ surpasses $T_{c_1}$, the SCCs divide the complex plane into three regions.  In comparison to the case of $T<T_{c_1}$, the additional region is small, formed by the self-encirclement of the SCCs (see Fig. \ref{fig18c}).  The stability regions $\Omega_T$, highlighted in green, shrinks as $T$ increases. As $T$ exceeds a critical value $T_{c_2}$, the directions of the SCCs are reversed, leading to the disappearance of the stability region (see Fig. \ref{fig18d}).

\begin{figure}
\begin{center}
\centering
\includegraphics[width=0.5\textwidth]{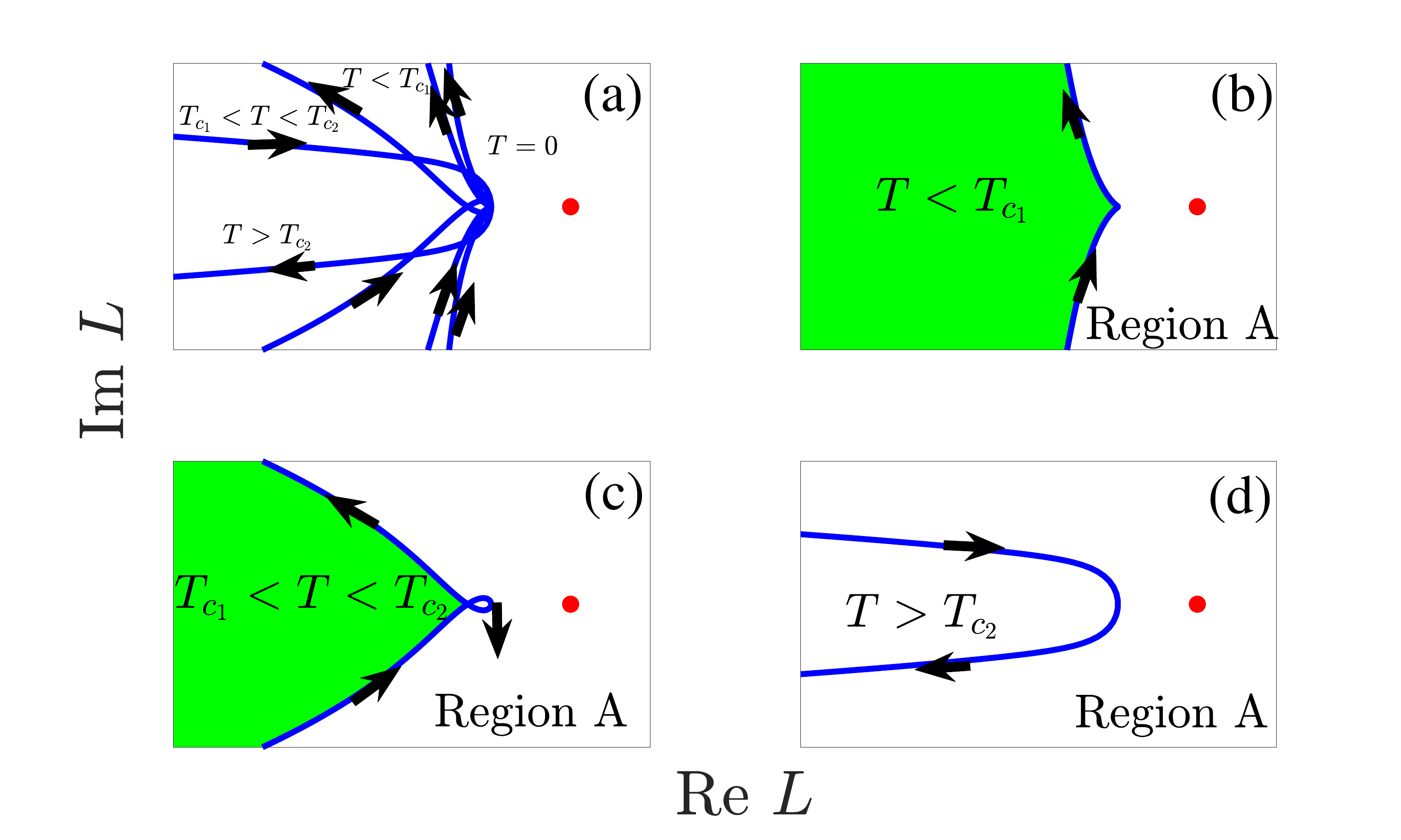}
\subfigure{\label{fig18a}}
\subfigure{\label{fig18b}}
\subfigure{\label{fig18c}}
\subfigure{\label{fig18d}}
\caption{ The SCCs and the corresponding stability regions for Eq.~\eqref{MAS2} change with $T$, where $a=b=k_1=1$ and $k_2=1.1$.
(a) The SCCs, determined by \eqref{MAS3}, own different shapes and directions for different values of $T\geq 0$.   The stability region $\Omega_T$, highlighted in green, is formed by different shapes of the SCCs for $T<T_{c_1}$ (b) and for $T_{c_1}<T<T_{c_2}$ (c). The stability region vanishes for $T>T_{c_2}$ (d). 
The black arrow indicates the increasing direction of the SCCs with respect to $\beta$. Here, the red dots represent the origin, and ${\rm NU}({\rm Region}~A)=1$ accordingly.  The value of ${\rm NU}$ for each region is computed using Theorem \ref{geometric}, leading to the stability (green-shaded) region $\Omega_T$ for Eq.~\eqref{MAS2}.} \label{fig18}
\end{center}
\end{figure}

Next, we are to seek the values of $T_{c_1,c_{2}}$. From the parametric representation in \eqref{MAS3},  the polar coordinates for the SCCs can be further obtained as 
\begin{equation}\label{MAS4}
\theta_T(\beta)=\pi+\arctan\dfrac{a\beta}{\beta^2+b}+\arctan \beta T-\arctan\dfrac{k_2\beta}{k_1}
\end{equation}
and
\begin{equation}\label{MAS5}
r_T(\beta)=\dfrac{\sqrt{(\beta^2+b)^2+a^2\beta^2}\sqrt{1+\beta^2T^2}}{\sqrt{k_1^2+k_2^2\beta^2}}.
\end{equation}
From the direction of the SCCs at $\beta=0$ (see Fig. \ref{fig18b} and Fig. \ref{fig18c}), we obtain the critical value $T_{c_1}=\frac{k_2}{k_1}-\frac{a}{b}=0.1$ from $\theta'_T(0)=0$.  On the other hand,  $\theta_T(\beta)=\pi+\left(a+\frac{k_1}{k_2}-\frac{1}{T}\right)\frac{1}{\beta}+o\left(\frac{1}{\beta}\right)$ as $\beta\to+\infty$.  Thus, the critical value $T_{c_2}\approx 0.5238$ from $a+\frac{k_1}{k_2}-\frac{1}{T}=0$. Hence, the contours of the stability region $\Omega_T$ can be parameterized as: 
$$
L=
\left\{\begin{array}{ll}
L_T(\beta),  & \beta\in\mathbb{R}, ~T\leq T_{c_1},\\
L_T(\beta),  &\beta\in (-\infty, -\beta^*]\cup[\beta^*,+\infty), ~ T_{c_1}<T<T_{c_2},
\end{array}\right.
$$
where $\beta^*>0$ is the root of $\theta_T(\beta^*)=\pi$ as $T_{c_1}<T<T_{c_2}.$

\begin{figure}
	\begin{center}
		\centering
		\includegraphics[width=0.45\textwidth]{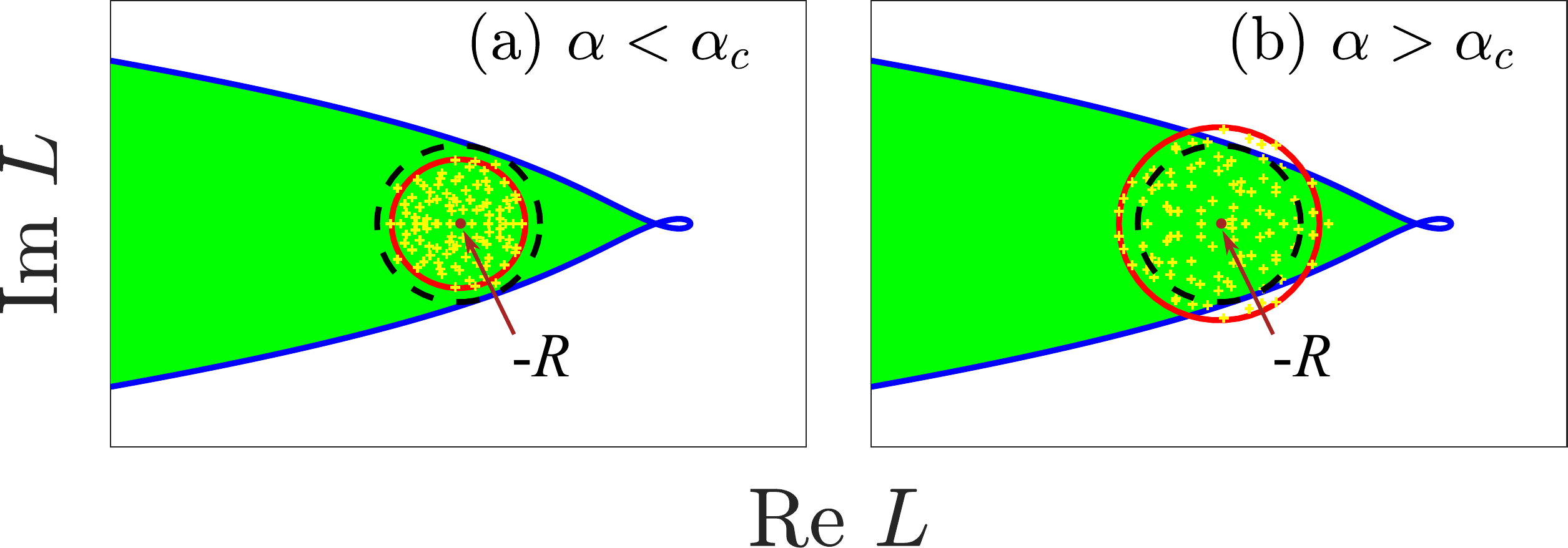}
		\subfigure{\label{figma}}
		\subfigure{\label{figmb}}
		\caption{The eigenvalue distribution of the random network matrix $\bm{J}=-R\bm{I}_N+\alpha\bm{\Xi}_N$ changes with $\alpha$.  For sufficiently large $N$, the eigenvalues of $\bm{J}$ (yellow crossings) are approximated as uniformly distributed within a circle centered at $-R$ (brown dots) obeying Eq. \eqref{1057} (red curves). As $\alpha$ increases from $0$, all eigenvalues are located within $\Omega_T$ (green-shaded regions) until it surpasses the critical value $\alpha_c$ determined by \eqref{1102}, where the circle \eqref{1057} becomes tangent to the SCCs (black dashed curves).} \label{figm}
	\end{center}
\end{figure}

Now, returning to the stability problem of system \eqref{MAS}, we consider a noise-perturbed self-negative feedback network $\bm{J}=-R\bm{I}_N+\alpha\bm{\Xi}_N$ where $\bm{\Xi}_N\triangleq\{\xi_{ij}\}_{N\times N}$, each $\xi_{ij}$ are independently sampled from a uniform distribution within the interval  $[-1,1]$, and 
$\alpha$ is the noise strength.  According to the Circular Law (see Section \ref{CL}, Theorem \ref{CL1}), for sufficiently large $N$,  the eigenvalues of $J = \{a_{ij}\}_{N\times N}$ are approximated as uniformly distributed within a circle obeying the equation
\begin{equation}\label{1057}
\left({\rm Re}~L+R\right)^2+\left({\rm Im}~L\right)^2=\dfrac{N\alpha^2}{3}.
\end{equation}
Fix $N$, $T\geq 0$, and $R>0$. When $\alpha=0$, all the eigenvalues of $\bm{J}$ collapse to a single point $-R$ (brown dots in Fig. \ref{figm}). Suppose that $-R\in\Omega_T$. As $\alpha$ increases from $0$, all eigenvalues are located within $\Omega_T$ until reaching a critical value $\alpha_c$, when the circle \eqref{1057} becomes tangent to the SCCs (see black dashed curves in Fig. \ref{figm}).  Here, the critical value $\alpha_c=\alpha_c(T,R,N)$ is obtained as follows:
\begin{equation}\label{1102}
\alpha_c=\sqrt{\dfrac{3}{N}}\inf_{\beta\in\mathbb{R}}|L_T(\beta)+R|.
\end{equation}
To numerically validate the analytically-obtained results, we choose four different values of $R$. For each value,  system \eqref{MAS} with controller \eqref{251} is numerically implemented for $1000$ times with the parameters $(\alpha,T)\in(0,4)\times (0,0.3)$. As shown in Fig. \ref{fign}, the analytical criteria in \eqref{1102} are confirmed by our numerical results.\hfill{$\square$}

\begin{figure}
	\begin{center}
		\centering
		\includegraphics[width=0.45\textwidth]{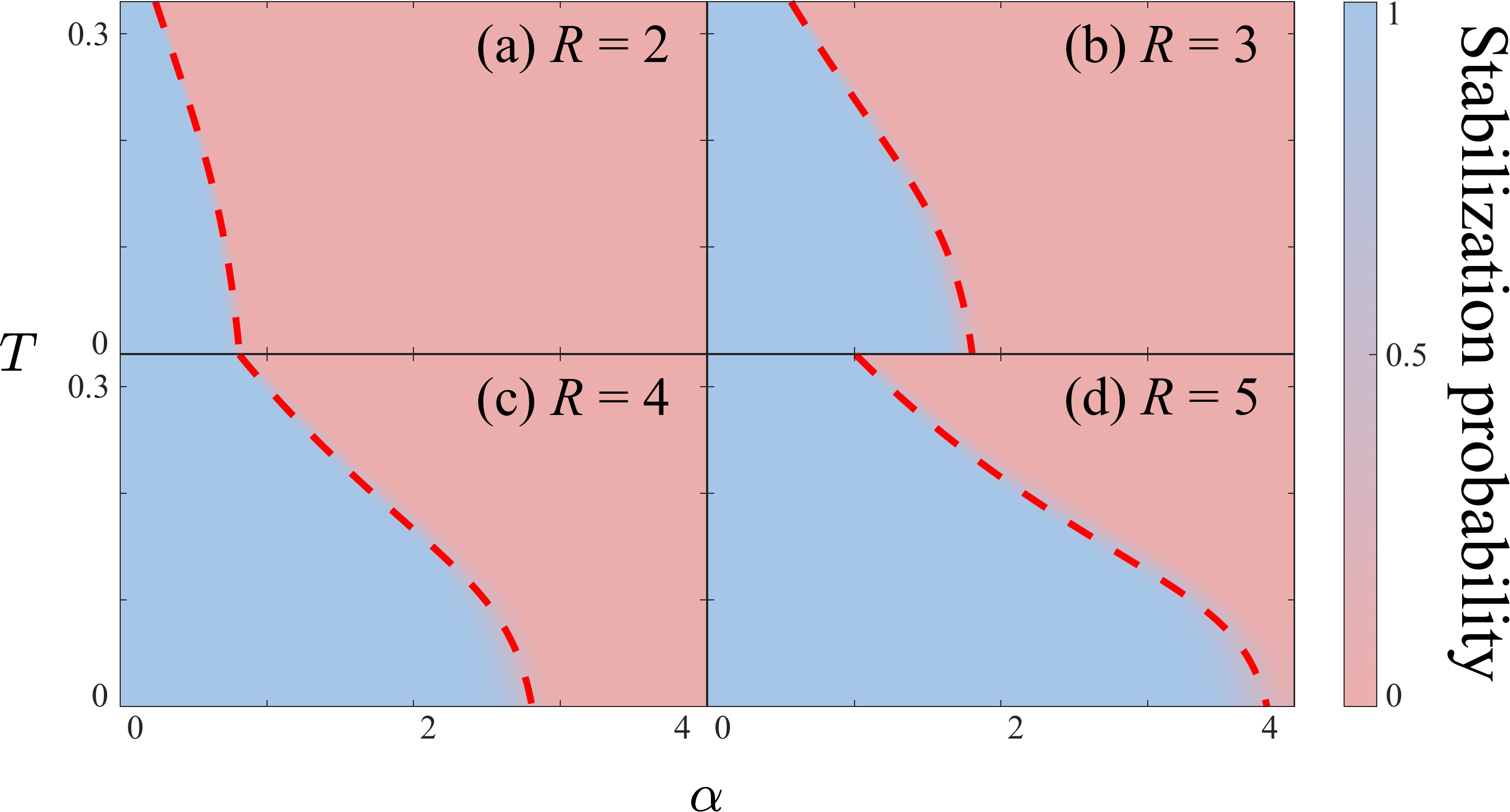}
		\subfigure{\label{figna}}
		\subfigure{\label{fignb}}
		\caption{For different values of $R$, depicted is the probability of successful stabilization for the controlled system \eqref{MAS} using the controller \eqref{251}, with different mean values $T$ of time delay distribution as well as different noise strengths $\alpha$ in random networks. The color of each point represents the frequency of successful stabilization from $1000$ numerical realizations. The red dashed curves correspond to the boundary obtained from \eqref{1102}. The other parameters are $a = b = k_1 = 1$, $k_2 = 1.1$, and $N = 100$.} \label{fign}
	\end{center}
\end{figure}




\end{example}

\begin{figure*}[h]
\begin{center}
\centering
\subfigure[]{\includegraphics[width=0.49\textwidth]{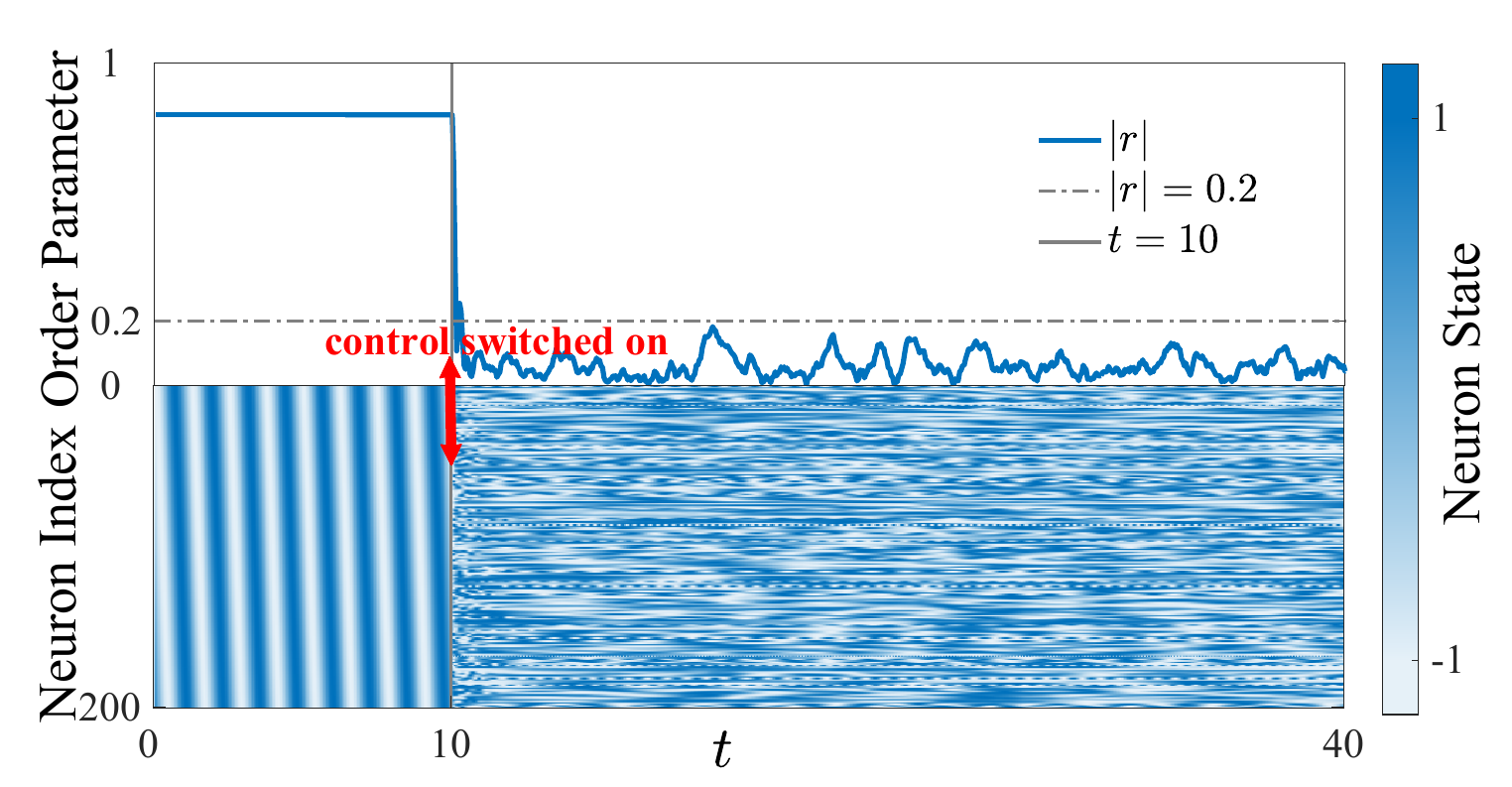}\label{fig10a}}
\subfigure[]{\includegraphics[width=0.49\textwidth]{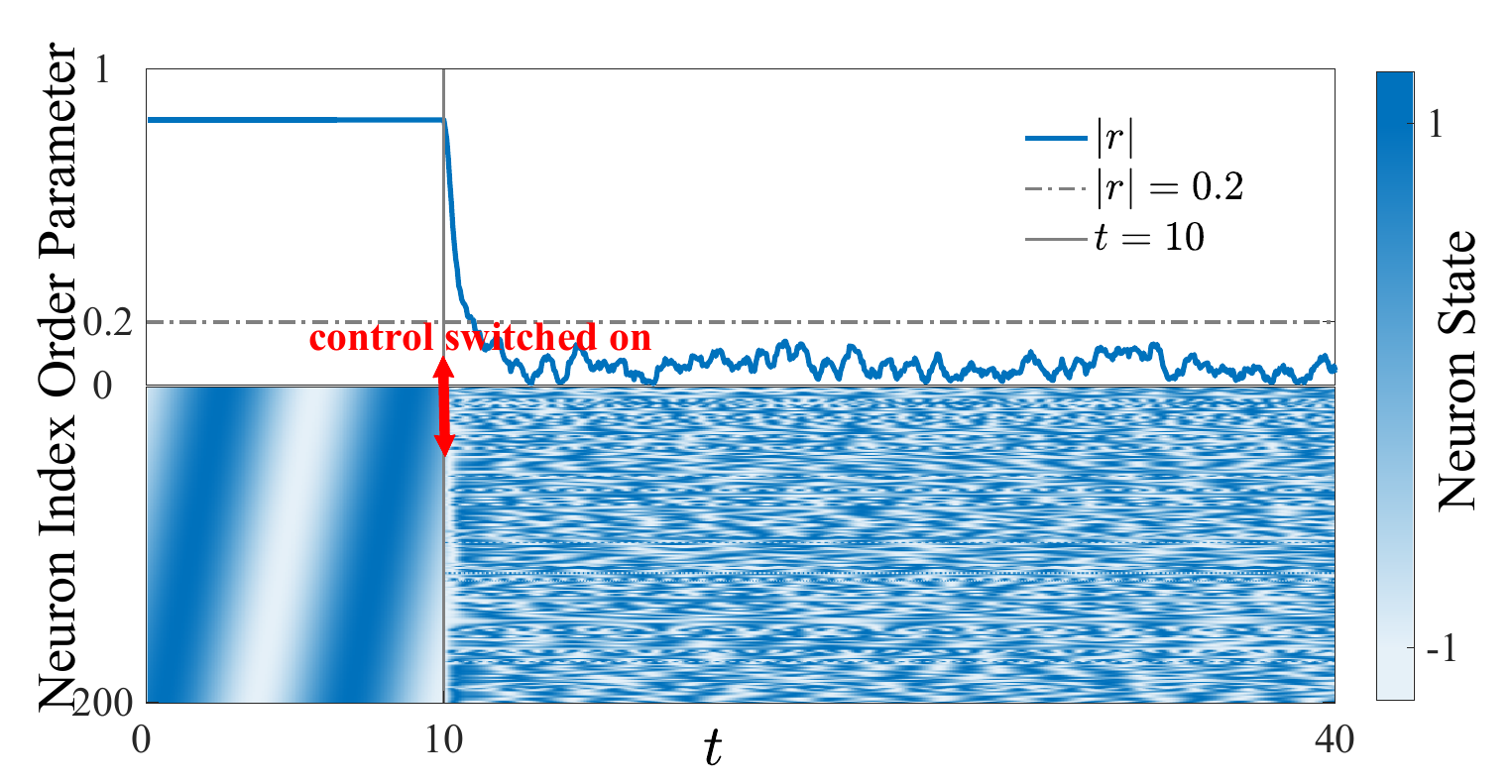}\label{fig10b}}
\caption{Dynamics for system \eqref{kuramoto},  the $200$ coupled Kuramoto's oscillators, with and without the controller \eqref{kuramoto2}.  The upper panels in (a,b) depict how the absolute value of the order parameter $r$, as defined in \eqref{centroid1}, changes with time.  The lower panels in (a,b) depict the individual dynamics for the $200$ coupled Kuramoto's oscillators.  Here, each color represents the sinusoidal value of the oscillator's phase.  The coupled oscillators show phase synchronization in the absence of control ($t<10$). When the feedback, with heterogeneous delays \eqref{kuramoto2} and suitable coupling gain, is switched on ($t>10$), the phase synchronization is eliminated. In (a), the parameters are taken as $d=0$, $K=4$, $C=-16$, $S=2$, and $\tau_{ij}$ obeys an exponential distribution with mean value $0.5$.  In (b),  the parameters are set as $d=2.5$, $K=4$, $C=-1$, $S=-3$, and $\tau_{ij}\equiv 0.5$. Additionally, the parameters $C$ and $S$ used in (a,b) correspond to the black dots in Figs. \ref{fig14} and \ref{fig12b}, respectively, which represent $L=(C+{\rm i}S)/2$.} \label{fig10}
\end{center}
\end{figure*}

The final example goes to an application of our proposed approach to  realizing synchronization elimination in a large population of coupled oscillators. The proposed controller in this example has potential use in deep brain stimulation, especially for the remedy of mental disorders including Parkinson's disease and epilepsy. The application background will be further discussed about in Remark \ref{background}.

\begin{example}\label{qutongbu}
We consider a system of coupled Kuramoto's oscillators described by:
\begin{equation}\label{kuramoto}
\dot{\theta}_i(t)=\omega_i+\dfrac{K}{N}\sum_{j=1}^N\sin[\theta_j(t)-\theta_i(t)]+u_i(t), ~ i=1,\cdots, N.
\end{equation}
Here, the $i$-th oscillator, represented by the phase dynamics $\theta_i$, is supposed to rotate on a unit circle \cite{b39,b56}.  Moreover, $\omega_i$ represents the natural frequency of the $i$-th oscillator, $K$ is the global coupling strength, $N$ is the total number of oscillators (assumed to be sufficiently large), and $u_i$ is the control input.  From a viewpoint of applications, all the natural frequencies are not necessarily be identical, which is supposed to obey a unimodal Cauchy-Lorentz distribution as
$$
g(\omega)=\dfrac{1}{\pi}\dfrac{1}{(\omega-d)^2+1}.
$$
As $K=0$, each oscillator rotates at its own natural frequency, showing desynchronization state.  For $K>2$, the coupled oscillators without control ($u_i(t)\equiv 0$) show a phenomenon of phase synchronization (see Fig. \ref{fig10} as $t<10$). In order to \textit{eliminate this phase synchronization}, we introduce feedback couplings with heterogeneous delays as:
\begin{equation}\label{kuramoto2}
\begin{array}{l}
\displaystyle u_i(t)=\dfrac{C}{N}\sum_{j\neq i}^N\sin[\theta_j(t-\tau_{ij})-\theta_i(t)] \\
\displaystyle  \hspace{2cm} +\dfrac{S}{N}\sum_{j\neq i}^N\cos[\theta_j(t-\tau_{ij})-\theta_i(t)].  
\end{array}
\end{equation}
Here, $C$ and $S$ are constant real-valued coupling strengths to be designed while the delay $\tau_{ij}$ means the time required to transfer a signal from the $j$-th oscillator to the $i$-th one. In real neuronal systems, oscillators are always spatially randomly and sparsely connected, which indicates that time delays obey a specific distribution $h(\tau)$ rather than being identical. Denote by \begin{equation}\label{centroid1}
r(t)\triangleq\dfrac{1}{N}\sum_{j=1}^N {\rm e}^{{\rm i}\theta_j(t)}
\end{equation}
the order parameter.  Using the classic mean-field method (see Remark \ref{Ott} and Section \ref{Ottantenson}), the macroscopic dynamics of $r(t)$ obeys:
\begin{equation}\label{kuramoto3}
\begin{aligned}
\dot{r}=\left(\dfrac{K}{2}-1+{\rm i}d\right)r+{L}\int_0^{+\infty}r(t-\tau)h(\tau){\rm d}\tau\\
-\dfrac{K}{2}|r|^2r-{\overline{L}}r^2 \int_0^{+\infty}\overline{r(t-\tau)}h(\tau){\rm d}\tau,
\end{aligned}
\end{equation}
where $L\triangleq (C+{\rm i}S)/2$ is the complex-valued coupling strength. Notice that $r(t)\equiv 0$ and $|r(t)|>0$ correspond, respectively, to the desynchronization state and the synchronization state (see Remark \ref{tongbu}). Hence, linearization of system \eqref{kuramoto3} in the vicinity of $r=0$ yields a time-delayed dynamical system
as 
\begin{equation}\label{kuramoto4}
\dot{r}=\left(\dfrac{K}{2}-1+{\rm i}d\right)r+{L}\int_0^{+\infty}r(t-\tau)h(\tau){\rm d}\tau.
\end{equation}
Therefore, our goal is to find a suitable complex-valued $L$ as the coupling gain to stabilize the zero solution of system \eqref{kuramoto4} for $K>2$. Such a goal has been realized for some particular cases in Section \ref{sec5.3}. In the following, we discuss about two special cases, viz., Case A: $\tau_{ij}\equiv T$ is a constant, i.e. $h(\tau)=\delta(\tau-T)$, and Case B: $d=0$ and $\tau_{ij}$ obeys an exponential distribution with mean value $T>0$, i.e. $h(\tau)=\frac{1}{T}{\rm e}^{-\frac{\tau}{T}}$.  For both cases, using the analytical arguments in a normal way yields a stability region for system \eqref{kuramoto4} if and only if $({K}/{2}-1)T<1$.  Actually, as $({K}/{2}-1)T<1$, appropriate values for the coupling strengths of $C$ and $S$ can be selected so that $L=(C+{\rm i}S)/2$ falls within the stability region (refer to Figs. \ref{fig14} and \ref{fig12b}).  The effectiveness of our proposed feedback with heterogeneous delays has been confirmed by the numerical results depicted in Fig. \ref{fig10}.  \hfill{$\square$}

\begin{remark}\label{background}
When degenerative neurons in the brain generate collective but abnormal oscillations, often occur the brain disorders such as Parkinson's disease and epilepsy. In the literature, several deep brain stimulation techniques have been developed  to treat these synchronization-induced mental disorders by eliminating synchronization in oscillatory neurons \cite{b75,b76}.  Actually, Example \ref{qutongbu} presents a mathematical model for addressing this problem.  For further details, refer to \cite{b40, b41, b53}.
\end{remark}

\begin{remark}\label{Ott}
The macroscopic dynamics \eqref{kuramoto3} for $r(t)$ is derived by using the mean-field method, specifically known as the Ott-Antonsen (OA) reduction method. We provide a concise brief about it in Section \ref{Ottantenson}. For a comprehensive understanding of the OA method, please refer to \cite{b39,b56}.
\end{remark}
\begin{remark}\label{tongbu}
The order parameter $r$ can be regarded as the \textit{centroid} of all oscillators. The values of $|r|$ vary in the interval $[0,1]$.  In the desynchronization state, the phases $\theta_i$ are uniformly distributed over the interval $[0,2\pi]$, which corresponds to a nearly zero value for $|r|$ (see Fig. \ref{fig10}, when $t>10$). Conversely, in the synchronization state, the phases $\theta_i$ are highly concentrated around a single value, leading to $|r|$ being close to $1$ (see Fig. \ref{fig10}, when $t<10$). Consequently, small values of $|r|$ signify the desynchronization state, whereas values of $|r|$ close to $1$ signify synchronization state.  Therefore, we opt to employ $|r|$  as a metric to describe the synchronization or desynchronization state for system \eqref{kuramoto}. For additional information, please refer to \cite{b39,b56}.
\end{remark} 
\end{example}
\section{Concluding remarks}\label{sec6}
Investigating the dynamics of MASs with complex networks and time delays has garnered significant interest in various real-world applications. Employing the master stability functions or the mean-field method allows us to convert these problems into lower-dimensional and linear time-delay systems.  Previous studies often use a prerequisite  that the network matrices are symmetric or/and the corresponding eigenvalues are all real.  This naturally invites the necessity of investigation on complex-valued TCEs when complex-valued eigenvalues are induced by asymmetric networks.  Additionally, previous studies also largely overlook the potential impact of memory effects in system dynamics, which arise from utilizing past information within a specific time interval. In this article, we have proposed a geometric approach for stability analysis of linear time-delay systems. \textcolor{black}{Our approach investigates the stability region of complex-valued parameters, which addresses general networks, including random and asymmetric networks.}  Furthermore, it incorporates various types of delays. Our approach allows for the design of delayed control strategies, enabling the achievement of consensus or the elimination of synchronization in multi-agent systems with general complex networks. To illustrate the effectiveness and practicality of our proposed approach, we have demonstrated its application in three representative examples.

There are still  a few unresolved open problems in this area of study.

Firstly, as illustrated in Figs. \ref{fig15} and \ref{fig18}, the stability region, denoted by $\Omega_T$, satisfies the condition $\Omega_{T_1}\subseteq\Omega_{T_2}$ for $T_1>T_2$. This implies that time delay has a negative impact on the consensus of systems \eqref{carfollowing2} and \eqref{MAS} for certain parameters. However, in some cases, time delay may actually enhance the consensus/stability in MASs under specific conditions. For example, if the parameters are set as $a=b=k_1=1$ and $k_2=2$ in system \eqref{MAS}, the stability region $\Omega_T$ does not shrink as $T$ increases.  As shown in Fig. \ref{fig26}, if some eigenvalues of the network matrix lie within the green-shaded regions, an increase in $T$ may lead to the consensus/stability even if it is not achieved at $T=0$. Some remarkable work has been done on the positive impact of time delay in \cite{b44,b49,b72}. It is worth examining this direction further in future studies.

\begin{figure}
\begin{center}
\centering
\includegraphics[width=0.5\textwidth]{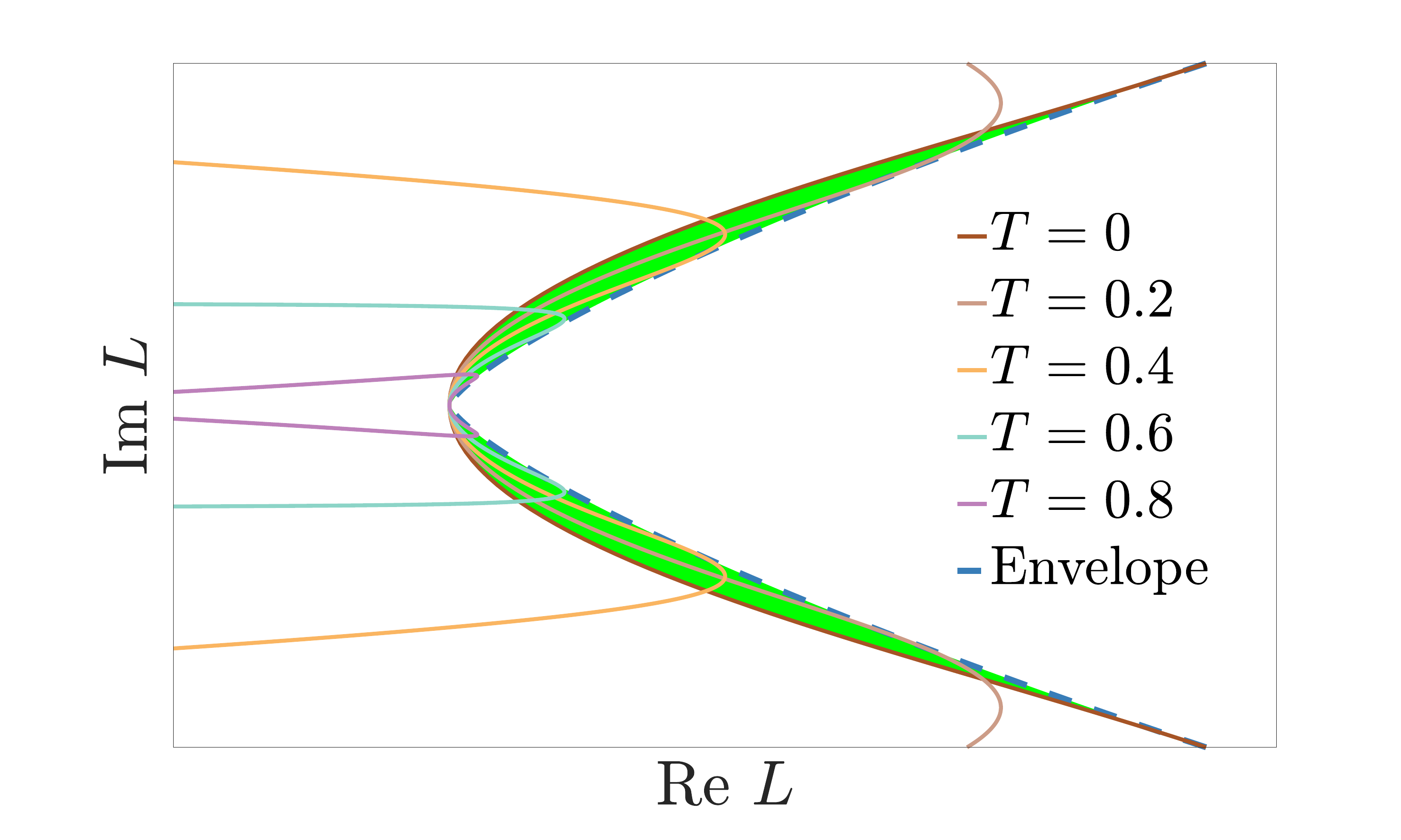}
\caption{The stability region  $\Omega_T$ for Eq. \eqref{MAS2} does not shrink as $T$ increases.  Here, the solid curves with different colors indicate the SCCs according to \eqref{MAS3} for different $T$, while the light blue dashed curve indicates the envelope of the SCC family. As some eigenvalues of the network matrix lie within the green-shaded regions, increasing $T$ may lead to the consensus/stability even if it is not achieved at $T=0$. Here, the parameters are set as $a=b=k_1=1$ and $k_2=2$. } \label{fig26}
\end{center}
\end{figure}

Secondly, it is worthwhile to mention that the root continuity argument (Theorem \ref{theorem1}) is {\it not} obvious, dependent on the specific form of the original system. To illustrate this, we investigate a scalar system with the delayed PD control protocol, which reads:
$$
\dot{z}=az+L[z(t-\tau)+\dot{z}(t-\tau)],
$$
which cannot be expressed in the form of system \eqref{2}. Applying Rouch\'{e}'s Theorem\cite[Chapter 3, Theorem 4.3]{b71}, when $|L|>1$, the TCE $(\lambda-a)/(1+\lambda)=L{\rm e}^{-\lambda\tau}$ has infinitely many roots in $\mathbb{C}_+$ {at infinity}. Consequently, as the parameter $L$ moves from $|L|\leq 1$ to $|L|>1$, or the parameter $\tau$ moves from $\tau=0$ to $\tau>0$, infinitely many roots emerge in $\mathbb{C}_+$ suddenly.  This implies that 
${\rm NU}(\cdot)$ \textit{may change its value even without roots appearing on the imaginary axis $\mathbb{C}_0$}.  This further indicates that these changes may not occur at the SCCs. In such scenarios, our geometric approach cannot be employed. It is valuable to explore scenarios in which our geometric approach can be employed when the original system is not in the form of system \eqref{2}.



Thirdly, our approach primarily centers on time-invariant systems. However, it is important to note that time delays in many real-world systems are actually time-varying. In recent research, remarkable progress has been made in the stability analysis of continuous-time systems with stochastic delays \cite{b57, b58,b73,b74}.  Thus, an application of our approach to investigate the time-varying systems is a prospective avenue for further study. 

Fourthly, our approach places particular emphasis on the stability of linear systems with control delays. Another interesting direction for future study involves solving the optimal control problem for linear distributed time-delay systems, which has been addressed in \cite{b59}. This direction also holds considerable potential for further investigation.

Finally, our approach focuses on ODE-based systems incorporating delays.  Notably, there has been significant research conducted on PDE-based systems with delays in recent times \cite{b77}. This certainly presents a promising avenue for future investigation as well.

\section{Appendix}

\subsection{Master stability function}\label{MSF}

In this subsection, we provide an overview of the master stability function (MSF)\cite{b54}. The MSF serves as a valuable tool in decomposing multi-agent systems into two distinct components: one that evolves alongside the synchronization manifold and another that evolves orthogonal to it. If the latter component demonstrates asymptotic stability, it guarantees the consensus/synchronization among the set of multi-agents. As a consequence, the MSF proves to be a potent instrument in investigating consensus within multi-agent systems.

Consider the following system with $N$ agents
\begin{equation}\label{01222103}
\dot{\bm{x}_i}=\bm{Q}\bm{x}_i+\sum_{j\neq i}^Na_{ij}[\bm{x}_j(t-\tau)-\bm{x}_i(t-\tau)],~ i=1,\cdots,N.
\end{equation}
Here, each $\bm{x}_i\in\mathbb{R}^q$ denotes the dynamics of agent $i$. It can be written as
\begin{equation}\label{01222104}
\dot{\bm{X}}=\bigg(\bm{I}_N \otimes \bm{Q}\bigg)\bm{X}+\bigg(\bm{J} \otimes \bm{I}_q\bigg) \bm{X}(t-\tau) \\
\end{equation}
where ``$\otimes$'' is the Kronecker product and $\bm{X}(t)\triangleq\left[\bm{x}_1^{\rm T}(t), \cdots, \bm{x}_N^{\rm T}(t)\right]^{\rm T}$. The diagonal elements of the network matrix $\bm{J}\triangleq\{a_{ij}\}_{N\times N}$ are defined as $a_{ii}\triangleq-\sum_{j=1, j\neq i}^Na_{ij}$.

We utilize a complex Schur transformation\cite[Theorem 8.9]{b70}, where a $N-$dimensional complex unitary matrix $\bm{P}$ is introduced such that $\bm{U}=\bm{P}^{-1}\bm{J}\bm{P}$, with $\bm{U}$ being upper triangular. The complex eigenvalues $\mu_1, \cdots, \mu_{N}$ of $\bm{J}$ are positioned along the main diagonal of $\bm{U}$. Note that the row sums of $\bm{J}$ are zero, at least one eigenvalue of $\bm{J}$ is zero. We assume that $\mu_1=0$. By introducing the transformation $\bm{Z}=\bigg(\bm{P} \otimes \bm{I}_q\bigg)^{-1} \bm{X}$, system \eqref{01222104} becomes
\begin{equation}\label{01222101}
\begin{aligned}
\dot{\bm{Z}}=&\bigg(\bm{P} \otimes \bm{I}_q\bigg)^{-1}\bigg(\bm{I}_N \otimes \bm{Q}\bigg)\bigg(\bm{P} \otimes \bm{I}_q\bigg)\bm{Z}\\
&+\bigg(\bm{P} \otimes \bm{I}_q\bigg)^{-1}\bigg(\bm{J} \otimes \bm{I}_q\bigg)\bigg(\bm{P} \otimes \bm{I}_q\bigg)\bm{Z}(t-\tau)\\
=&\bigg(\bm{I}_N \otimes \bm{Q}\bigg)\bm{Z}+\bigg(\bm{U} \otimes \bm{I}_q\bigg)\bm{Z}(t-\tau).
\end{aligned}
\end{equation}
Due to the block-diagonal structure of $\bm{I}_N \otimes \bm{Q}$ and the upper triangular structure of $\bm{U}$, the stability of system \eqref{01222101} is equivalent to the stability of the subsystems
\begin{equation}\label{01222102}
\dot{\bm{z}}_k=\bm{Q}\bm{z}_k+\mu_k\bm{z}_k(t-\tau), ~~ k=1,\cdots,N.
\end{equation}
Here, $\bm{Z}(t)\triangleq\left[\bm{z}_1^{\rm T}(t), \cdots, \bm{z}_N^{\rm T}(t)\right]^{\rm T}$. Since $\mu_1=0$, we have
$\dot{\bm{z}}_1=\bm{Q}\bm{z}_1$
evolves along the synchronization manifold. In contrast, system \eqref{01222102} with $k=2, \cdots, N$ evolves transversely to the synchronization manifold \cite{b54}.  Therefore, the multi-agent systems achieve the consensus (i.e. $\lim_{t\to+\infty}\|\bm{x}_i(t)-\bm{x}_j(t)\|=0$ for $1\leq i<j\leq N$) if and only if system \eqref{01222102} with $k=2, \cdots, N$ is stable.  In other words, the consensus can be achieved if all the eigenvalues of the matrix $\bm{J}$ (except for $0$) are located in the stability region $$\Omega\triangleq \left\{L\in\mathbb{C}~\Big|~\dot{\bm{z}}=\bm{Q}\bm{z}+L\bm{z}(t-\tau)~\mbox{is stable} \right\}.$$

\subsection{Mean-field method} \label{Ottantenson}

In this subsection, we provide an overview of the mean-field method, which is commonly known as the OA reduction method\cite{b39,b56}, to obtain the macroscopic dynamics \eqref{kuramoto3} for $r(t)$ using system \eqref{kuramoto} with the controller \eqref{kuramoto2}.

In the continuum limit $N \rightarrow \infty$, the state of the oscillator system at time $t$ can be described by a continuous distribution function $f(\omega, \theta, t)$, in terms of frequency $\omega$ and phase $\theta$, for the problems in system \eqref{kuramoto}, where
$$
\int_0^{2 \pi} f(\omega, \theta, t){\rm d}\theta=g(\omega).
$$
In this case, the ``mean-field'' order parameter is written as:
\begin{equation}\label{1648}
r(t)=\int_{-\infty}^{\infty} \int_0^{2 \pi} f(\omega, \theta, t) {\rm e}^{{\rm i} \theta} {\rm d}\theta {\rm d}\omega.
\end{equation}
Note that system \eqref{kuramoto} can be expressed in terms of the order parameter as:
$$
\dot{\theta}_i(t)=\omega_i+K\operatorname{Im}\left[r{\rm e}^{-{\rm i} \theta_i(t)}\right]+u_i(t), ~~ i=1,2,\cdots, N.
$$
Meanwhile, the controller \eqref{kuramoto2} can be expressed in terms of an ``order parameter" $\eta_i$ as:
\begin{equation}\label{1539}
u_i(t)=C\operatorname{Im}\left[\eta_i {\rm e}^{-{\rm i} \theta_i(t)}\right]+S\operatorname{Re}\left[\eta_i {\rm e}^{-{\rm i} \theta_i(t)}\right],
\end{equation}
where
$$
\eta_i(t)\triangleq N^{-1} \sum_{j=1}^N {\rm e}^{{\rm i}\theta_j(t-\tau_{ij})} .
$$
We assume that all delays $\tau_{ij}$ obey the distribution $h(\tau)$ and are uncorrelated with the oscillator frequencies $\omega$ at either end of the link. Therefore, we obtain
\begin{equation}\label{1709}
\eta_i(t)\approx\eta(t)\triangleq\int_0^{\infty} r(t-\tau) h(\tau) {\rm d}\tau.
\end{equation}
Using the continuity equation, we obtain the evolution of oscillator distribution function $f(\theta, \omega, t)$ as follows:
\begin{equation}\label{1631}
\begin{aligned}
\frac{\partial}{\partial t} f+\frac{\partial}{\partial \theta}\Bigg\{\Big[\omega+&\frac{K}{2{\rm i}}\left({\rm e}^{-{\rm i}\theta} r-{\rm e}^{{\rm i}\theta}\overline{r}\right)\\
+\frac{C}{2{\rm i}}\left({\rm e}^{-{\rm i} \theta}\eta-{\rm e}^{{\rm i} \theta}\overline{\eta}\right)+&\frac{S}{2}\left({\rm e}^{-{\rm i}\theta}\eta+{\rm e}^{{\rm i}\theta}\overline{\eta}\right)\Big] f\Bigg\}=0.
\end{aligned}
\end{equation}
Writing $f(\omega, \theta, t)$ in a Fourier series yields:
$$
f(\omega, \theta, t)=\frac{g(\omega)}{2 \pi}\left\{1+\sum_{n=1}^{\infty}\left[f_n(\omega, t) {\rm e}^{{\rm i}n\theta}+\overline{f_n(\omega, t)}{\rm e}^{-{\rm i}n\theta}\right]\right\}.
$$
Following the method outlined in \cite{b39,b56}, we consider the dynamics \eqref{1631} on an invariant manifold:
\begin{equation}\label{1632}
f_n(\omega, t)=[a(\omega, t)]^n .
\end{equation}
The macroscopic dynamics of $a(\omega, t)$ is derived by substituting Eq.~\eqref{1632} into Eq.~\eqref{1631}, which further leads to
\begin{equation}\label{1651}
\dfrac{\partial a}{\partial t}+{\rm i}\omega a+\dfrac{K}{2}\left(r a^2-\overline{r}\right)+\left({L}\eta a^2-\overline{L\eta}\right)=0,
\end{equation}
where $L\triangleq (C+{\rm i}S)/2$. The oscillator frequency distribution $g(\omega)$ is assumed to obey
$g(w)=\frac{1}{\pi}\frac{1}{(w-d)^2+1}$.
We assume that $a(\omega,t)$ is analytic with respect to $\omega$ in the lower half complex plane. By the Residue Formula\cite[Chapter 3, Theorem 2.1]{b71}, we obtain that
\begin{equation}\label{1652}
r(t)=\int_{-\infty}^{\infty} g(\omega) \overline{a(\omega, t)} {\rm d}\omega=\overline{a(d-{\rm i},t)}.
\end{equation}
Here, the first equality is obtained by substituting Eq.~\eqref{1632} into Eq.~\eqref{1648}. Furthermore, by setting $\omega\triangleq d-{\rm i}$ and $a\triangleq a(d-{\rm i},t)$ in Eq.~\eqref{1651}, we have
\begin{equation}\label{1653}
\dfrac{{\rm d}a}{{\rm d}t}+{\rm i}(d-{\rm i})a+\dfrac{K}{2}\left(r a^2-\overline{r}\right)+\left({L}\eta a^2-\overline{L\eta}\right)=0.
\end{equation}
Substituting Eqs.~\eqref{1709} and \eqref{1652} into Eq.~\eqref{1653}, and then taking the conjugate of both sides give  Eq.~\eqref{kuramoto3}.

\subsection{Circular Law}\label{CL}

In this subsection, we provide an overview of the Circular Law, which is a useful tool for approximating the eigenvalue distribution of large random network matrices.

\begin{theorem}\label{CL1}
(Circular Law, \cite[Theorem 1.10]{b100}) Consider an $N\times N$ complex random matrix $\Xi_N$ whose entries are mutually independent and identically distributed copies of a complex random variable with zero mean and finite variance $\sigma^2$.   Further let $\hat{\mu}_1,...,\hat{\mu}_N$ be the eigenvalues of $\hat{\bm{\Xi}}_N \triangleq \bm{\Xi}_N/\sigma\sqrt N$. 
The empirical spectral distribution (ESD) $\delta_N$ of $\hat{\bm{\Xi}}_N$ is defined as:
\begin{equation*}
\delta_N(x,y) \triangleq \frac{1}{N} \#\left\{k\leq N ~\big|~ {\rm Re}\{\hat{\mu}_k\} \leq x,~ {\rm Im}\{\hat{\mu}_k\}\leq y\right\}.
\end{equation*}
As $N\to+\infty$, the ESD $\delta_N$ converges, both in probability and in an almost sure sense, to the uniform distribution on the unit disk $\delta_{\text{cir}}$, defined as:
\begin{equation*}
	\delta_{\text{cir}}(x,y) \triangleq \frac{1}{\pi} \text{mes}\left(
	\left\{z\in \mathbb{C}~\big|~ |z|\leq 1, ~{\rm Re}\{z\} \leq x,~ {\rm Im} \{z\}\leq y\right\}\right).
\end{equation*}
\end{theorem}

In Example \ref{MAS0}, each entry $\xi_{ij}$ of the matrix $\bm{\Xi}_N$ is independently sampled from a uniform distribution within the interval $[-1,1]$, which indicates that $\sigma^2=1/3$. According to Theorem \ref{CL1}, for sufficiently large $N$, the eigenvalues of $\hat{\bm{\Xi}}_N \triangleq \bm{\Xi}_N/\sigma\sqrt N$ are approximated as uniformly distributed in the unit circle $\left\{z\in \mathbb{C}~\big|~|z|=1\right\}$. Furthermore, this implies that the eigenvalues of $\bm{J}=-R\bm{I}_N+\alpha\sigma\sqrt{N}\hat{\bm{\Xi}}_N $ are approximated as uniformly distributed within a circle obeying Eq. \eqref{1057}.

\bibliographystyle{IEEEtran}
\bibliography{brief_revised}

\begin{thebibliography}{10}
\providecommand{\url}[1]{#1}
\csname url@samestyle\endcsname
\providecommand{\newblock}{\relax}
\providecommand{\bibinfo}[2]{#2}
\providecommand{\BIBentrySTDinterwordspacing}{\spaceskip=0pt\relax}
\providecommand{\BIBentryALTinterwordstretchfactor}{4}
\providecommand{\BIBentryALTinterwordspacing}{\spaceskip=\fontdimen2\font plus
\BIBentryALTinterwordstretchfactor\fontdimen3\font minus
  \fontdimen4\font\relax}
\providecommand{\BIBforeignlanguage}[2]{{%
\expandafter\ifx\csname l@#1\endcsname\relax
\typeout{** WARNING: IEEEtran.bst: No hyphenation pattern has been}%
\typeout{** loaded for the language `#1'. Using the pattern for}%
\typeout{** the default language instead.}%
\else
\language=\csname l@#1\endcsname
\fi
#2}}
\providecommand{\BIBdecl}{\relax}
\BIBdecl

\bibitem{b29}
H.~D. Unbehauen, \emph{Control Systems, Robotics and Automation--Volume XIII:
  Nonlinear, Distributed, and Time Delay Systems-II}.\hskip 1em plus 0.5em
  minus 0.4em\relax EOLSS Publications, 2009.

\bibitem{b30}
Y.~He, M.~Wu, J.-H. She, and G.-P. Liu, ``Parameter-dependent Lyapunov
  functional for stability of time-delay systems with polytopic-type
  uncertainties,'' \emph{IEEE Transactions on Automatic Control}, vol.~49,
  no.~5, pp. 828--832, 2004.

\bibitem{b31}
T.~H. Lee and J.~H. Park, ``A novel Lyapunov functional for stability of
  time-varying delay systems via matrix-refined-function,'' \emph{Automatica},
  vol.~80, pp. 239--242, 2017.

\bibitem{b32}
D.~R. Reddy, A.~Sen, and G.~L. Johnston, ``Time delay effects on coupled limit
  cycle oscillators at Hopf bifurcation,'' \emph{Physica D: Nonlinear
  Phenomena}, vol. 129, no. 1-2, pp. 15--34, 1999.

\bibitem{b33}
Q.~Gao and J.~Ma, ``Chaos and Hopf bifurcation of a finance system,''
  \emph{Nonlinear Dynamics}, vol.~58, pp. 209--216, 2009.

\bibitem{b34}
T.~Li and Q.~Wang, ``Stability and Hopf bifurcation analysis for a two-species
  commensalism system with delay,'' \emph{Qualitative Theory of Dynamical
  Systems}, vol.~20, pp. 1--20, 2021.

\bibitem{b21}
J.~K. Hale and S.~M.~V. Lunel, \emph{Introduction to Functional Differential
  Equations}.\hskip 1em plus 0.5em minus 0.4em\relax Springer Science \&
  Business Media, 2013.

\bibitem{b42}
S.-I. Niculescu, \emph{Delay Effects on Stability: A Robust Control
  Approach}.\hskip 1em plus 0.5em minus 0.4em\relax Springer Science \&
  Business Media, 2001.

\bibitem{b43}
K.~Gu, J.~Chen, and V.~L. Kharitonov, \emph{Stability of Time-Delay
  Systems}.\hskip 1em plus 0.5em minus 0.4em\relax Springer Science \& Business
  Media, 2003.

\bibitem{b2}
X.-G. Li, S.-I. Niculescu, A.~Cela, L.~Zhang, and X.~Li, ``A frequency-sweeping
  framework for stability analysis of time-delay systems,'' \emph{IEEE
  Transactions on Automatic Control}, vol.~62, no.~8, pp. 3701--3716, 2017.

\bibitem{b14}
M.~S. Lee and C.~Hsu, ``On the $\tau$-decomposition method of stability
  analysis for retarded dynamical systems,'' \emph{SIAM Journal on Control and
  Optimization}, vol.~7, no.~2, pp. 242--259, 1969.


 \bibitem{b79}
N.~Olgac and R.~Sipahi, ``An exact method for the stability analysis of time-delayed linear time-invariant (LTI) systems," \emph{IEEE Transactions on Automatic Control}, vol.~47, no.~5, pp. 793-797,  2002. 
  
\bibitem{b80}  
W.~Qiao and R.~Sipahi, ``A linear time-invariant consensus dynamics with homogeneous delays: Analytical study and synthesis of rightmost eigenvalues,'' \emph{SIAM Journal on Control and Optimization}, vol.~51, no.~5, pp. 3971-3992, 2013.
  
  
\bibitem{b81} 
  Q.~Gao and N.~Olgac, ``Stability analysis for LTI systems with multiple time delays using the bounds of its imaginary spectra,'' \emph{Systems \& Control Letters}, vol.~102, pp. 112-118, 2017.  
  
\bibitem{b12}
J.~Chen, P.~Fu, S.-I. Niculescu, and Z.~Guan, ``An eigenvalue perturbation
  approach to stability analysis, Part I: Eigenvalue series of matrix
  operators,'' \emph{SIAM Journal on Control and Optimization}, vol.~48, no.~8,
  pp. 5564--5582, 2010.

\bibitem{b13}
J.~Chen, P.~Fu, S.-I. Niculescu, and Z.~Guan, ``An eigenvalue perturbation approach to stability analysis, Part II:
  When will zeros of time-delay systems cross imaginary axis?'' \emph{SIAM
  Journal on Control and Optimization}, vol.~48, no.~8, pp. 5583--5605, 2010.




\bibitem{b27}
E.~N. Gryazina and B.~T. Polyak, ``Stability regions in the parameter space:
  D-decomposition revisited,'' \emph{Automatica}, vol.~42, no.~1, pp. 13--26,
  2006.
  
  
\bibitem{b25}
Y.~I. Neimark, ``Determination of the values of parameters for which an
  automatic system is stable,'' \emph{Avtomatika i Telemekhanika}, vol.~9, pp.
  190--203, 1948.




\bibitem{b26}
J.~I. Nejmark, ``D-decomposition of the space of quasipolynomials (on the stability of linearized distributive systems),''
  \emph{American Mathematical Society Translations}, vol.~102, pp. 95--131, 1973.
  
  
  
  


\bibitem{b22}
M.~Lichtner, M.~Wolfrum, and S.~Yanchuk, ``The spectrum of delay differential
  equations with large delay,'' \emph{SIAM Journal on Mathematical Analysis},
  vol.~43, no.~2, pp. 788--802, 2011.

\bibitem{b23}
J.~Sieber, M.~Wolfrum, M.~Lichtner, and S.~Yanchuk, ``On the stability of
  periodic orbits in delay equations with large delay,'' \emph{ArXiv Preprint
  ArXiv:1101.1197}, 2011.

\bibitem{b16}
S.~Yanchuk, M.~Wolfrum, T.~Pereira, and D.~Turaev, ``Absolute stability and
  absolute hyperbolicity in systems with discrete time-delays,'' \emph{Journal
  of Differential Equations}, vol. 318, pp. 323--343, 2022.

\bibitem{b17}
X.~Li, H.~Gao, and K.~Gu, ``Delay-independent stability analysis of linear
  time-delay systems based on frequency discretization,'' \emph{Automatica},
  vol.~70, pp. 288--294, 2016.

\bibitem{b24}
F.~Brauer, ``Absolute stability in delay equations,'' \emph{Journal of
  Differential Equations}, vol.~69, no.~2, pp. 185--191, 1987.

\bibitem{b15}
K.~Gu, S.-I. Niculescu, and J.~Chen, ``On stability crossing curves for general
  systems with two delays,'' \emph{Journal of Mathematical Analysis and
  Applications}, vol. 311, no.~1, pp. 231--253, 2005.

\bibitem{b18}
Q.~An, E.~Beretta, Y.~Kuang, C.~Wang, and H.~Wang, ``Geometric stability switch
  criteria in delay differential equations with two delays and delay dependent
  parameters,'' \emph{Journal of Differential Equations}, vol. 266, no.~11, pp.
  7073--7100, 2019.

\bibitem{b19}
E.~Beretta and Y.~Kuang, ``Geometric stability switch criteria in delay
  differential systems with delay dependent parameters,'' \emph{SIAM Journal on
  Mathematical Analysis}, vol.~33, no.~5, pp. 1144--1165, 2002.

\bibitem{b57}
D.~Antunes and H.~Qu, ``Frequency-domain analysis of networked control systems
  modeled by markov jump linear systems,'' \emph{IEEE Transactions on Control
  of Network Systems}, vol.~8, no.~2, pp. 906--916, 2021.

\bibitem{b58}
D.~Antunes, ``Frequency-domain analysis of aperiodic control loops with
  identically distributed delays,'' \emph{Automatica}, vol. 151, p. 110626,
  2023.

\bibitem{b7}
X.-G. Li, S.-I. Niculescu, and A.~Cela, ``Complete stability of linear
  time-delay systems: A new frequency-sweeping frequency approach,'' \emph{2013
  10th IEEE International Conference on Control and Automation}, pp.
  1121--1126, 2013.

\bibitem{b8}
X.-G. Li, S.-I. Niculescu, A.~Cela, H.-H. Wang, and T.-Y. Cai, ``Invariance
  properties for a class of quasipolynomials,'' \emph{Automatica}, vol.~50,
  no.~3, pp. 890--895, 2014.

\bibitem{b9}
A.~Mesbahi and M.~Haeri, ``Stability of linear time invariant fractional delay
  systems of retarded type in the space of delay parameters,''
  \emph{Automatica}, vol.~49, no.~5, pp. 1287--1294, 2013.

\bibitem{b36}
P.~Appeltans, S.-I. Niculescu, and W.~Michiels, ``Analysis and design of
  strongly stabilizing PID controllers for time-delay systems,'' \emph{SIAM
  Journal on Control and Optimization}, vol.~60, no.~1, pp. 124--146, 2022.

\bibitem{b37}
D.~Ma, I.~Boussaada, J.~Chen, C.~Bonnet, S.-I. Niculescu, and J.~Chen, ``PID
  control design for first-order delay systems via MID pole placement:
  Performance vs. robustness,'' \emph{Automatica}, vol. 137, p. 110102, 2022.

\bibitem{b28}
X.-G. Li, S.-I. Niculescu, and A.~Cela, \emph{Analytic Curve Frequency-Sweeping
  Stability Tests for Systems with Commensurate Delays}.\hskip 1em plus 0.5em
  minus 0.4em\relax Springer, 2015.

\bibitem{b47}
Y.-J. Chen, X.-G. Li, G.-X. Fan, and Y.~Zhang, ``Consensus for a class of
  multi-agent systems with time delays: A systematic study in parameter
  space,'' \emph{IEEE Transactions on Automatic Control}, vol.~69, no.~3, pp. 1769--1803, 2024.

\bibitem{b48}
W.~Hou, M.~Fu, H.~Zhang, and Z.~Wu, ``Consensus conditions for general
  second-order multi-agent systems with communication delay,''
  \emph{Automatica}, vol.~75, pp. 293--298, 2017.

\bibitem{b38}
X.-G. Li, S.-I. Niculescu, J.-X. Chen, and T.~Chai, ``Characterizing PID
  controllers for linear time-delay systems: a parameter-space approach,''
  \emph{IEEE Transactions on Automatic Control}, vol.~66, no.~10, pp.
  4499--4513, 2021.

\bibitem{b50}
D.~Ma, J.~Chen, and T.~Chai, ``Role of integral control for enlarging
  second-order delay consensus margin under PID protocols: None,'' \emph{IEEE
  Transactions on Cybernetics}, vol.~52, no.~11, pp. 11874--11884, 2022.

\bibitem{b51}
D.~Ma, J.~Chen, R.~Lu, J.~Chen, and T.~Chai, ``Delay effect on first-order
  consensus over directed graphs: optimizing PID protocols for maximal
  robustness,'' \emph{SIAM Journal on Control and Optimization}, vol.~60,
  no.~1, pp. 233--258, 2022.
  

  
  
  

\bibitem{b4}
X.-G. Li, S.-I. Niculescu, A.~{\c{C}}ela, and L.~Zhang, ``Stability analysis of
  uniformly distributed delay systems: A frequency-sweeping approach,''
  \emph{Delays and Interconnections: Methodology, Algorithms and Applications},
  pp. 117--130, 2019.

\bibitem{b6}
L.~Zhang, Z.-Z. Mao, X.-G. Li, S.-I. Niculescu, and A.~{\c{C}}ela, ``Stability
  analysis for a class of distributed delay systems with constant coefficients
  by using a frequency-sweeping approach,'' \emph{IET Control Theory \&
  Applications}, vol.~13, no.~1, pp. 87--95, 2019.

\bibitem{b5}
L.~Zhang, X.-G. Li, Z.-Z. Mao, J.-X. Chen, and G.-X. Fan, ``Some new algebraic
  and geometric analysis for local stability crossing curves,''
  \emph{Automatica}, vol. 123, p. 109312, 2021.

\bibitem{b35}
C.-U. Choe, R.-S. Kim, H.~Jang, P.~H{\"o}vel, and E.~Sch{\"o}ll,
  ``Delayed-feedback control: arbitrary and distributed delay-time and
  noninvasive control of synchrony in networks with heterogeneous delays,''
  \emph{International Journal of Dynamics and Control}, vol.~2, pp. 2--25,
  2014.

\bibitem{b45}
Y.-J. Chen, X.-G. Li, Y.~Zhang, S.-I. Niculescu, and A.~Cela, ``Stability
  analysis of car-following systems with uniformly distributed delays using
  frequency-sweeping approach,'' \emph{IEEE Access}, vol.~9, pp.
  69747--69755, 2021.

\bibitem{b46}
R.~Sipahi, F.~M. Atay, and S.-I. Niculescu, ``Stability of traffic flow
  behavior with distributed delays modeling the memory effects of the
  drivers,'' \emph{SIAM Journal on Applied Mathematics}, vol.~68, no.~3, pp.
  738--759, 2008.

\bibitem{b55}
R.~E.~Baker and G.~Röst, ``Global dynamics of a novel delayed logistic equation arising
  from cell biology,'' \emph{Journal of Nonlinear Science}, vol.~30, no.~1, pp. 397--418, 2020.

\bibitem{b54}
L.~M.~Pecora and T.~L.~Carroll, ``Master stability functions for synchronized coupled
  systems,'' \emph{Physical Review Letters}, vol.~80, no.~10,  pp. 2109-2112, 1998.

\bibitem{b60}
J.~C. Mitchell, ``Social networks,'' \emph{Annual Review of Anthropology},
  vol.~3, no.~1, pp. 279--299, 1974.

\bibitem{b61}
K.~Avrachenkov and N.~Litvak, ``The effect of new links on google pagerank,''
  \emph{Stochastic Models}, vol.~22, no.~2, pp. 319--331, 2006.

\bibitem{b62}
G.~Karlebach and R.~Shamir, ``Modelling and analysis of gene regulatory
  networks,'' \emph{Nature Reviews Molecular Cell Biology}, vol.~9, no.~10, pp.
  770--780, 2008.

\bibitem{b10}
W.~Lin, Y.~Pu, Y.~Guo, and J.~Kurths, ``Oscillation suppression and
  synchronization: Frequencies determine the role of control with time
  delays,'' \emph{Europhysics Letters}, vol. 102, no.~2, p. 20003, 2013.

\bibitem{b11}
J.~Nishiguchi, ``Stability region and critical delay,'' \emph{ArXiv Preprint
  ArXiv: 2109.10426}, 2021.

\bibitem{b40}
S.~Zhou, P.~Ji, Q.~Zhou, J.~Feng, J.~Kurths, and W.~Lin, ``Adaptive elimination
  of synchronization in coupled oscillator,'' \emph{New Journal of Physics},
  vol.~19, no.~8, p. 083004, 2017.

\bibitem{b41}
S.~Zhou and W.~Lin, ``Eliminating synchronization of coupled neurons adaptively
  by using feedback coupling with heterogeneous delays,'' \emph{Chaos: An
  Interdisciplinary Journal of Nonlinear Science}, vol.~31, no.~2, p. 023114, 2021.

\bibitem{b53}
K.~Wang, L.~Yang, S.~Zhou, and W.~Lin, ``Desynchronizing oscillators coupled in
  multi-cluster networks through adaptively controlling partial networks,''
  \emph{Chaos: An Interdisciplinary Journal of Nonlinear Science}, vol.~33,
  no.~9,  p. 091101, 2023.

\bibitem{b39}
E.~Ott and T.~M. Antonsen, ``Low dimensional behavior of large systems of
  globally coupled oscillators,'' \emph{Chaos: An Interdisciplinary Journal of
  Nonlinear Science}, vol.~18, no.~3, p. 037113, 2008.

\bibitem{b56}
W.~S. Lee, E.~Ott, and T.~M. Antonsen, ``Large coupled oscillator systems with
  heterogeneous interaction delays,'' \emph{Physical Review Letters}, vol. 103,
  no.~4, p. 044101, 2009.

\bibitem{b44}
Q.~Ma and S.~Xu, ``Consensus switching of second-order multiagent systems with
  time delay,'' \emph{IEEE Transactions on Cybernetics}, vol.~52, no.~5, pp.
  3349--3353, 2022.

\bibitem{b49}
Q.~Ma and S.~Xu, ``Intentional delay can benefit consensus of second-order multi-agent
  systems,'' \emph{Automatica}, vol. 147, p. 110750, 2023.

\bibitem{b59}
J.~M.~Ortega-Martínez,  O.~J.~Santos-Sánchez, L.~Rodríguez-Guerrero, and  S.~Mondie, ``On optimal control
  for linear distributed time-delay systems,'' \emph{Systems \& Control
  Letters}, vol. 177, p. 105548, 2023.

\bibitem{b70}
S.~Roman, S.~Axler, and F.~W.~Gehring,  \emph{Advanced Linear Algebra}. \hskip 1em plus 0.5em minus 0.4em\relax Springer, 2005.

\bibitem{b71}
E.~M.~Stein and R.~Shakarchi,  \emph{Complex Analysis}. \hskip 1em plus 0.5em minus 0.4em\relax Princeton University Press, 2010.

\bibitem{b72}
Z.~Meng, Z.~Li, A.~V.~Vasilakos, and S.~Chen, ``Delay-induced synchronization of identical linear multiagent systems," \emph{IEEE Transactions on Cybernetics}, vol.~43, no.~2, pp.~476--489, 2013

\bibitem{b73}
K.~Sijia and D.~Bresch-Pietri,  ``Prediction-based controller for linear systems with stochastic input delay,'' \emph{Automatica}, vol. 138, p. 110149, 2022.

\bibitem{b74}
J.~Auriol, K.~Sijia, and D.~Bresch-Pietri, ``Explicit prediction-based control for linear difference equations with distributed delays,"  \emph{IEEE Control Systems Letters}, vol.~6, pp.~ 2864--2869, 2022.

\bibitem{b75}
W.~Pasillas-Lépine, I.~Haidar, A.~Chaillet, and E.~Panteley, ``Closed-loop deep brain stimulation based on firing-rate regulation,"  \emph{2013 6th International IEEE/EMBS Conference on Neural Engineering}, pp.~166--169, 2013.


\bibitem{b76}
A.~Franci, A.~Chaillet, E.~Panteley, and  F.~Lamnabhi-Lagarrigue, ``Desynchronization and inhibition of Kuramoto oscillators by scalar mean-field feedback,''  \emph{Mathematics of Control, Signals, and Systems}, vol.~24, no.~1--2, pp.~169--217, 2012.

 
 
 \bibitem{b77}
 I.~Karafyllis, P.~Pepe, A.~Chaillet, and Y.~Wang,  ``Is global asymptotic stability necessarily uniform for time-invariant time-delay systems?''  \emph{SIAM Journal on Control and Optimization}, vol.~60, no.~6, pp. 3237--3261, 2022. 
 
 

 
 \bibitem{b78}
R.~Olfati-Saber and R.~M.~Murray, ``Consensus problems in networks of agents with switching topology and time-delays,"  \emph{IEEE Transactions on Automatic Control}, vol.~49, no.~9, pp. 1520--1533, 2004.
 
 
 \bibitem{b100}
 T.~Tao, V.~Vu, and M.~Krishnapur. ``Random matrices: Universality of ESDs and the circular law,'' \emph{The Annals of Probability}, vol.~38, no.~5, pp. 2023--2065, 2010.
 
 
 
 
 
 

\end{thebibliography}


{\small

\begin{wrapfigure}{l}[0cm]{0pt}
   \includegraphics[width=2.5cm]{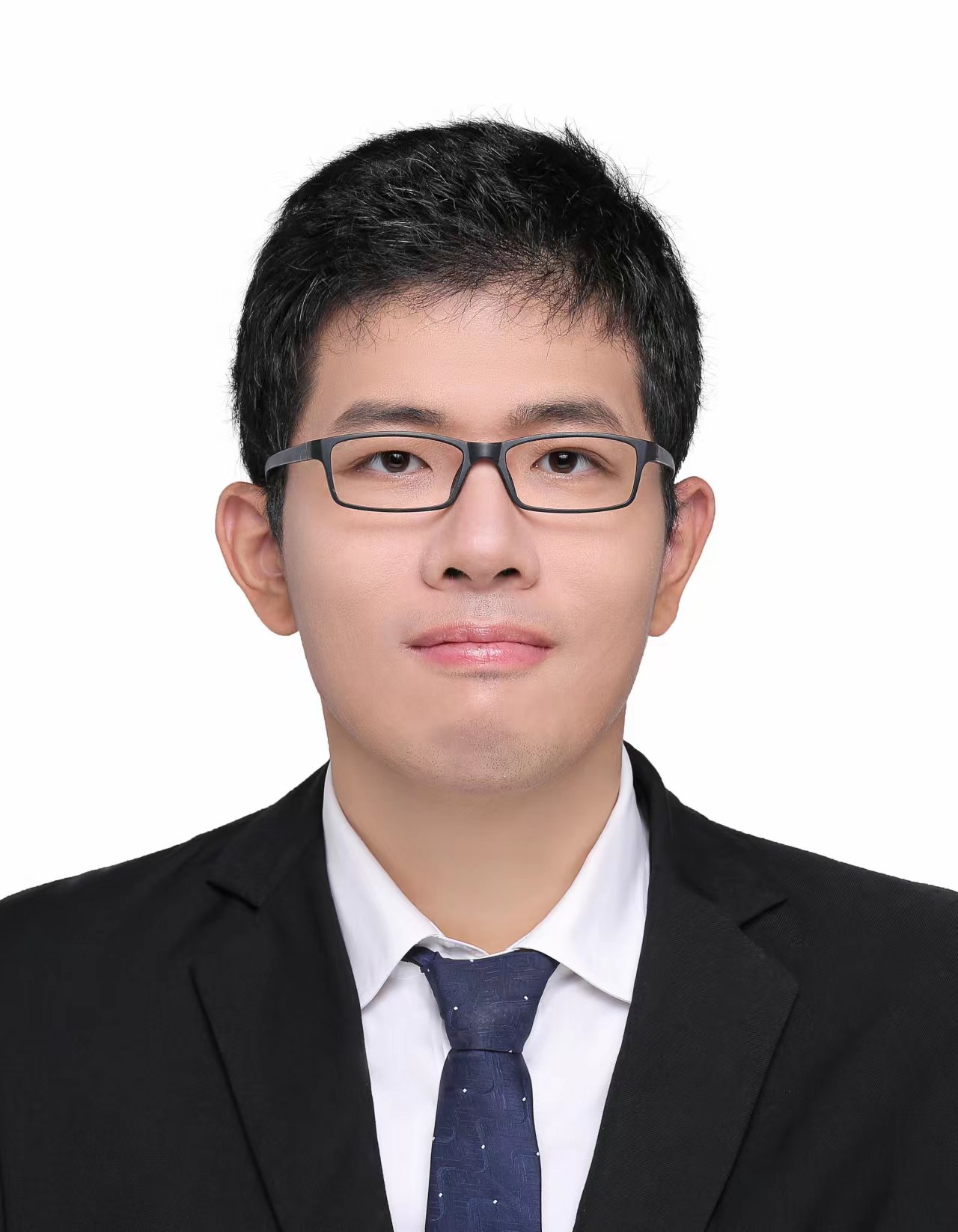}
\end{wrapfigure}

\noindent\textbf{Shijie Zhou} was born in Shanghai, China, in 1992. He received the B.S. and the Ph.D. degrees in applied mathematics from Fudan University, Shanghai, China, in 2014 and 2019, respectively. He was a postdoctor researcher in York University in Canada from 2020 to 2022.
He currently is a young investigator at the Research Institute of Intelligent Complex Systems, Fudan University, Shanghai, China. 

His current research interests include complex networks, stochastic differential equations, randomly-switching systems, and their applications to computational neuroscience.   His contributions have been published in prestigious journals in IEEE, SIAM, AIP, IOP, and Physical Review.

~~

~
~

~~
~~

~
~

~~

\begin{wrapfigure}{l}[0cm]{0pt}
   \includegraphics[width=2.5cm]{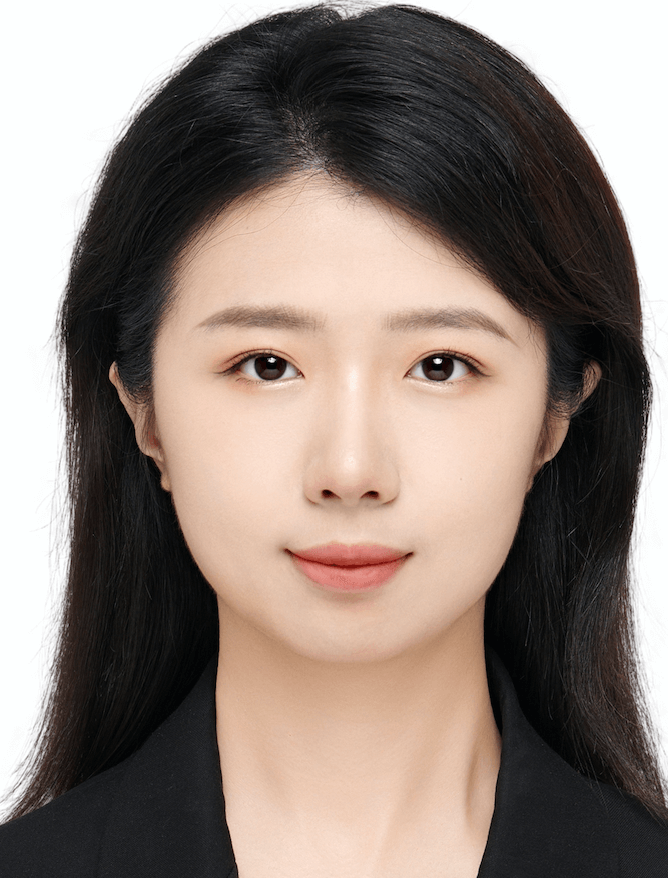}
\end{wrapfigure}

 \noindent\textbf{Luan Yang} received her M.S. degree in applied mathematics from Fudan University, Shanghai, in 2023. She is currently pursuing her Ph.D. degree at the Research Institute of Intelligent Complex Systems, Fudan University. Her research is primarily focused on complex systems, neural dynamics, and neuroscience. Her scholarly work has been published in esteemed journals under AIP and IEEE.

~~

~
~

~~
~~
~~

~
~

~~
~~
~
~

~~
~~
~
~

~~

~~

\begin{wrapfigure}{l}[0cm]{0pt}
   \includegraphics[width=2.5cm]{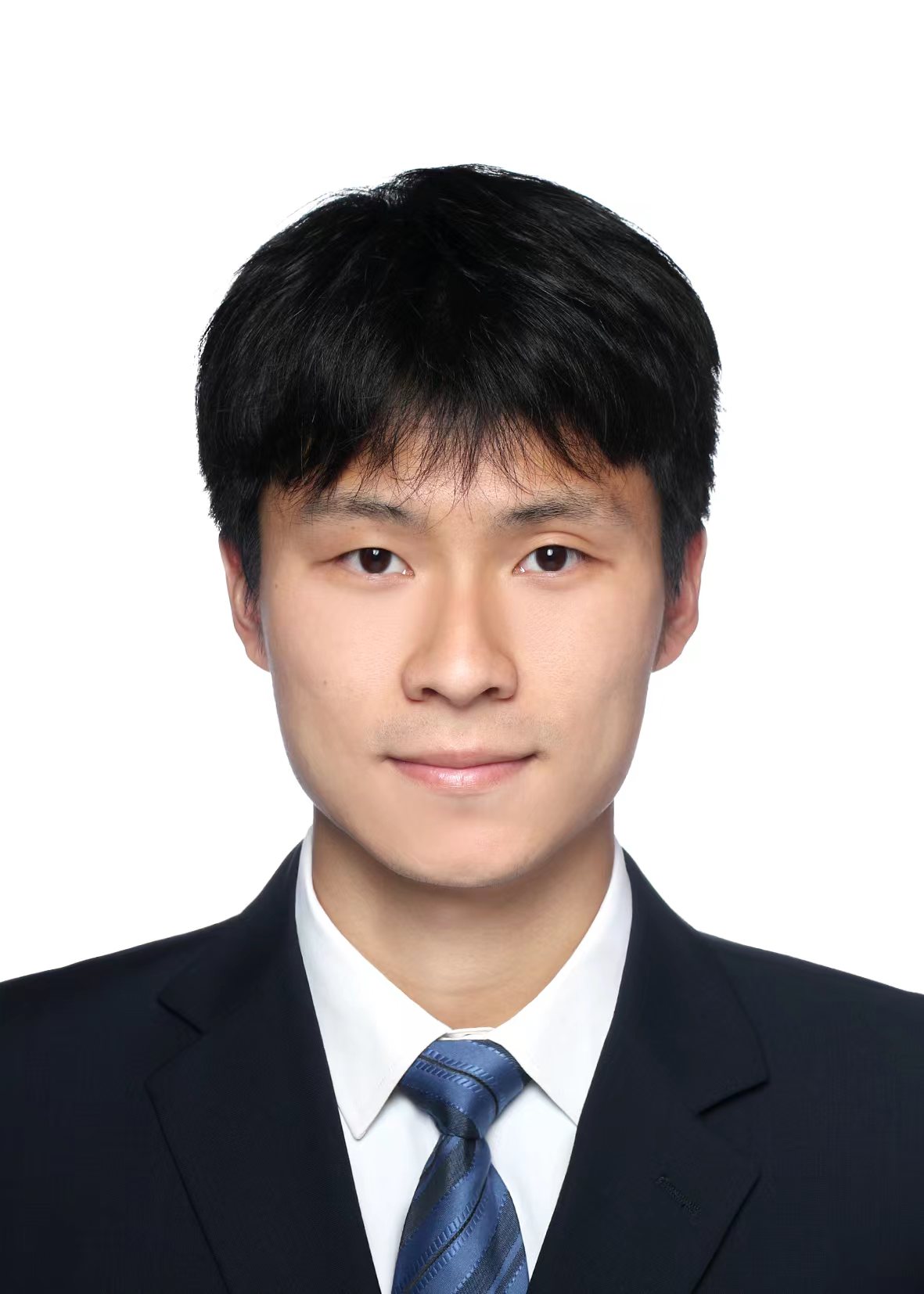}
\end{wrapfigure}

 \noindent\textbf{Xuzhe Qian}  received his B.S. degree in applied mathematics from Fudan University, Shanghai, China, in 2021. He is currently a Ph.D. student at the School of Mathematical Sciences, Fudan University. His research interests include  the collective behavior of complex systems, delay dynamical systems and computational social sciences.

~~

~
~

~
~

~~

~~

~~

~
~

~~

\begin{wrapfigure}{l}[0cm]{0pt}
\includegraphics[width=2.5cm]{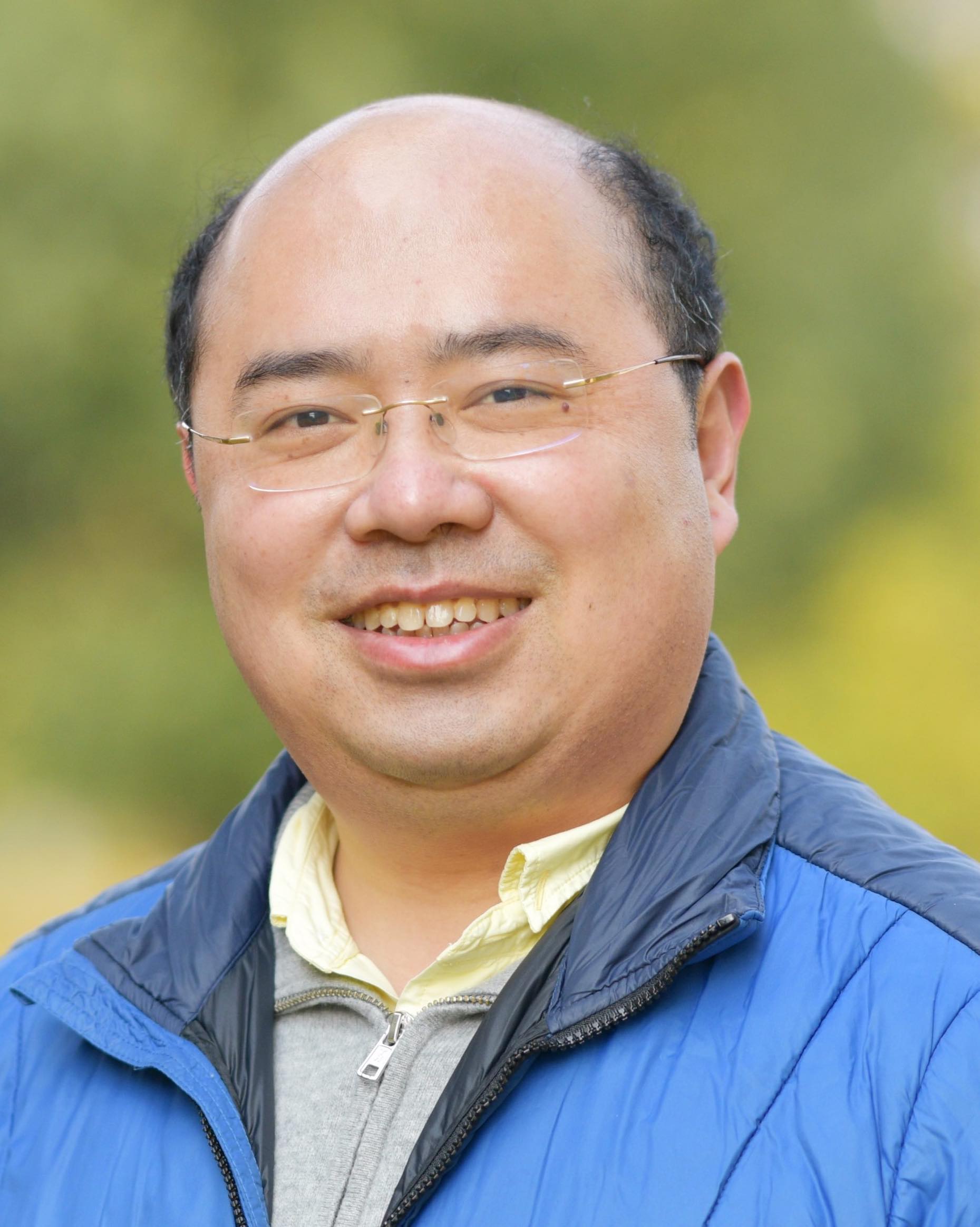}
\end{wrapfigure}

\noindent\textbf{Wei Lin}  received the Ph.D. degree in applied mathematics from Fudan University, Shanghai, China, in January 2003, with a specialization in nonlinear dynamical systems and artificial neural networks.

Since December 2009, he has been a Full Professor in applied mathematics with Fudan University. He is currently serving as the Dean of the Research Institute of Intelligent Complex Systems, the Vice Dean of the School of Data Science, and the Director of the Centre for Computational Systems Biology, Fudan University. From 2008 to 2013, he held a staff scientist position with the CAS-MPG Partner Institute for computational biology, Shanghai, China. His current research interests include bifurcation and chaos theory, stability and oscillations in hybrid systems, stochastic systems and complex networks, data assimilation, causality analytics, and all their applications to computational systems biology and artificial intelligence. His major contributions have been published in prestigious journals and conference proceedings, including PRL, PNAS, Nature Communications, Nature Physics, IEEE TAC, IEEE TNN, SIAM Journal on Control and Optimization, ICLR, NeurIPS, and AAAI.

Dr. Lin is the Senior Member of IEEE, the Vice Chair of the CSIAM Committee on Mathematical Life Science, the General Secretary of the Shanghai Society of Nonlinear Sciences, the Board Member of the International Physics and Control Society, the AE or Editor for IJBC/Research/CSF, and the member of the Editorial Advisory Board of CHAOS. He received the Excellent Young Scholar Fund and the Outstanding Young Scholar Fund from NSFC in 2013 and 2019, respectively, and he was selected as the Chief Scientist of the National Key R\&D Program of China. He was awarded as a Highly Cited Chinese Researcher in General Engineering according to Elsevier from 2015 to 2019.
He was a recipient of the Best Paper Prize from the International Consortium of Chinese Mathematicians in 2019, a second recipient of the First Prize of the Shanghai Natural Science Awards in 2020, and a recipient of the V. Afraimovich Award for outstanding young scholars in Nonlinear Physical Science in 2024.

~

}





~~



~
~~

~~

~~

~~

~~

~~


\end{document}